\documentclass[12pt]{amsart}

\usepackage{amsmath}
\usepackage{amssymb}
\usepackage{amsfonts}
\usepackage{mathrsfs}
\usepackage{graphicx}
\usepackage{epsfig}

\newcommand{\hU}{\widehat{U}}
\newcommand{\cU}{\mathcal{U}}
\newcommand{\cV}{\mathcal{V}}
\newcommand{\hQ}{\widehat{Q}}

\theoremstyle{plain}
\newtheorem{thm}{Theorem}[section]
\newtheorem{cor}[thm]{Corollary}
\newtheorem{lem}[thm]{Lemma}
\newtheorem{prop}[thm]{Proposition}
\newtheorem{remark}[thm]{Remark}

\newtheorem*{nonametheorem}{Main Theorem}
\newtheorem*{reducedmaintheorem}{Reduced Main Theorem}

\theoremstyle{definition}
\newtheorem{definition}{Definition}[section]

\newcommand{\dom}{\textrm{D}}

\newcommand{\eps}{\varepsilon}

\renewcommand{\l}{\lambda}

\begin{document}

\title[]{The topological complexity of Cantor attractors for unimodal interval maps}
\author{Simin Li and Weixiao Shen}
\address{Department of Mathematics,
University of Science and Technology of China, Hefei,
230026,  CHINA (e-mail:lsm@ustc.edu.cn)}
\address{Department of Mathematics, National University of Singapore, Block S17, 10 Lower Kent Ridge Road, Singapore 119076 (e-mail:matsw@nus.edu.sg)}
\date{}

\begin{abstract}For a non-flat $C^3$ unimodal map with a Cantor attractor,
we show that for any open cover $\mathcal U$ of this attractor, 
the complexity function $p(\mathcal U, n)$ is of order $n\log n$. In the appendix, we construct a non-renormalizable map with a Cantor attractor for which $p(\mathcal{U}, n)$ is bounded from above for any open cover $\mathcal{U}$.
\end{abstract}

\maketitle

\section{Introduction}
In this paper, we will consider the topological complexity of a unimodal interval map $f:[0,1]\to [0,1]$ restricted to an invariant Cantor set $X$. For an open cover $\cU$ of $X$, let $N(\cU)$ be the minimal cardinality of a sub-cover of $\cU$. For open covers $\cU, \cV$ of $X$, let $\cU \bigvee \cV=\{U\cap V: U\in \cU, V\in \cV\}$. The {\it topological complexity function} of an open cover $\cU$ is the non-decreasing function
$$p(\cU, n)= N(\bigvee_{i=0}^{n-1}f^{-i}\cU), \,\, n=1,2,\ldots.$$
Recall that the topological entropy of $f:X\to X$ is by definition
$$h_{\text{top}}(f|_X)=\sup_{\mathcal{U}} \lim_{n\to\infty} \frac{1}{n}\log p(\cU, n),$$
where the supremum is taken over all open covers of $X$.
The complexity functions can be used to characterize the dynamical behavior of some systems with topological entropy zero. For example, it is proved in \cite[Proposition 2.2]{BHM} that a system is equicontinuous if and only if the complexity function is bounded for each open cover.

A continuous map $f:[0,1] \to [0,1]$ is called  {\it unimodal} if there exists a unique $c\in (0,1)$ (called the {\em turning point}) such that $f$ is strictly increasing on $[0,c]$ and strictly decreasing on $[c,1]$.
In order to apply a convenient version of the Koebe principle, we shall always assume that $f$ is $C^3$ outside $c$, $f'(x)\not =0$ for $x\not =c$, and $c$ is non-flat, i.e., there exist $C^3$ local diffeomorphisms $\phi$, $\psi$ defined on a neighborhood of $0$ with $\phi(0)=c$, $\psi(0)=f(c)$, and a real number $\ell=\ell_c>1$ (called the {\it order} of $c$), such that $|\psi^{-1}\circ f\circ \phi(x)|=|x|^\ell$ holds when $|x|$ is small.
The turning point $c$ is also called a {\em critical point.}

Let $\mathcal{A}$ denote the collection of unimodal maps with the above properties and let $\mathcal{A}_*$ denote the collection of $f\in\mathcal{A}$ which have all periodic points hyperbolic repelling.

We are interested in the case that $X$ is a Cantor attractor. Following \cite{Milnor}, a (minimal) {\it metric attractor} is a compact invariant subset $X\subset [0,1]$ such that $\textrm{Rel}(X):=\{x\in [0,1]: \omega(x)\subset A\}$ has positive Lebesgue measure, but no invariant compact proper subset of $X$ has this property.  Metric attractors were studied in \cite{BL} under an additional assumption that $f$ has negative Schwarzian derivative, although most of their work extends to maps in the class $\mathcal{A}$ after \cite{K}, see also \cite{GSS}. In particular, it was shown that a metric attractor of $f\in\mathcal{A}$ can be one of the following forms: an attracting periodic orbit, or the union of a cycle of periodic intervals, or a Cantor set. In the last case, the Cantor attractor $X$ must coincide with $\omega(c)$ and $h_{\text{top}}(f|_X)=0,$ see ~\cite[Section 11]{BL}.

Our main result is the following theorem.

\begin{nonametheorem}
Let $f\in\mathcal{A}_*$ be a unimodal map with critical point $c$. Suppose that $\omega(c)$ is a Cantor attractor. Then for each open cover $\cU$ of $\omega(c)$, there is a constant $C>0$ such that  the complexity function of $f|\omega(c)$ satisfies $p(\cU,n)\le Cn\log n$ for $n>1$.
\end{nonametheorem}

One may wonder whether $p(\cU, n)$ has a lower bound for some (small) open cover $\cU$.
It is well-known that when $f$ is infinitely renormalizable, $\omega(c)$ is a Cantor attractor, and $f:\omega(c)\to\omega(c)$ is topologically conjugate to an adding machine and hence $p(\cU,n)$ is bounded for each open cover $\cU$ of $\omega(c)$.   
Even in the non-renormalizable case, there exists a unimodal map with a Cantor attractor for which $f:\omega(c)\to\omega(c)$ is again topologically conjugate to an adding machine,
as we show in Theorem~\ref{thm:ec} in \S\ref{sec:appendix}.
On the other hand, in Corollary~\ref{cor:special}, we prove that for interval maps with {\em special combinatorics} (including the well-studied Fibonacci case), $p(\cU, n)/n$ is bounded away from zero for small open covers of $\omega(c)$.

By considering open covers formed by nice intervals and their entry domains, we reduce the Main Theorem to an estimate of number of children of symmetric nice intervals. See the Reduced Main Theorem in \S\ref{subsec:nicecover}.

A Cantor attractor of non-infinitely renormalizable map $f\in\mathcal{A}$ is often called a {\em wild attractor} because its basin of attraction is of the first Baire category.
Existence of wild attractors for unimodal maps with the Fibonacci combinatorics was obtained in~\cite{BKNS}. This result was generalized in~\cite{B} to unimodal maps with ``Fibanacci-like'' combinatorics. While a sufficient and necessary combinatorial condition compatible with existence of wild attractor seems far from being reached, the dynamics of unimodal maps restricted to wild attractors was studied in~\cite{BKP, BSS,li-shen-Haudoff}, among others. Our construction in Theorem~\ref{thm:ec} is motivated by \cite{BKM} and the proof uses a result of \cite{B}.

\noindent {\bf Entropy zero systmes.}
There have appeared quite a few notions to measure the complexity  of topological dynamical systems of zero topological entropy. In the following, we shall mention two of them.
In \cite{DHP}, a notion called {\em topological entropy dimension} was introduced. (The metric entropy dimension was introduced earlier in \cite{FP}.) A topological dynamical system $f:X\to X$ has zero (upper) topological entropy dimension if for every open cover $\mathcal{U}$ and $\alpha>0$,
$n^{-\alpha}\log p_n(\mathcal{U})\to 0$ as $n\to\infty$. So our main theorem clearly implies that the topological entropy dimension of $f|\omega(c)$ is zero.

Another notion we would like to mention is the {\em topological sequence entropy} introduced in \cite{Goodman}. For an increasing sequence $(n_k)_{k=1}^\infty$ of positive integers, the sequence entropy of $f: X\to X$ is
$$h(T , (n_k)_{k=1}^{\infty}) = \sup_{\mathcal{U}}\limsup_{k\to\infty} \frac{1}{n_k}\log N\left (\bigvee_{i=1}^k T^{-n_i}(\mathcal{U})\right),$$
where the supremum is taken over all open covers of $X$. There are systems which have zero topological entropy but positive topological sequence entropy. A system is called a {\em null system} if the topological sequence entropy is zero for every sequence $(n_k)_{k=1}^\infty$. We shall see that the dynamics in a Cantor attractor is not necessarily null. In fact, in \cite[Theorem 3 (2)]{BKP}, Bruin, Keller and Pierre constructed a unimodal map $f\in\mathcal{A}_*$ together with a symbolic dynamical system $(\Omega, T)$ such that
\begin{itemize}
\item $(\Omega, T)$ is minimal, uniquely ergodic and weakly mixing respect to its unique invariant probability measure $\mu$;
\item $f$ has a wild attractor $\omega(c)$;
\item $f|\omega(c)$ is a factor of $(\Omega, T)$.
\end{itemize}

\begin{thm}
For the above example of Bruin-Keller-Pierre, $f:\omega(c)\to \omega(c)$ has zero entropy dimension but is not null.
\end{thm}
\begin{proof}
By the Main Theorem, $f:\omega(c)\to\omega(c)$ has entropy dimension zero. Let us show that $f:\omega(c)\to\omega(c)$ is not null. Indeed, since $T$ is minimal, for each non-empty open set $U\subset \Omega$, $\Omega=\bigcup_{n=0}^{+\infty} T^{-n}U$ and therefore $\mu(U)>0$. 
Since $T$ is weakly mixing with respect to $\mu$, $T\times T$ is ergodic with respect to $\mu\times \mu$. As $\mu\times \mu$ has positive measure on each non-empty open set, it follows that $T\times T$ is topologically transitive, see \cite[Theorem 5.16]{W}.
Thus $$\widehat{f}:=f\times f: \omega(c)\times \omega(c)\to \omega(c)\times \omega(c)$$ is topologically transitive.
Arguing by contradiction, assume that $f|\omega(c)$ is null. Then by \cite[Theorem 4.3]{HLSY}, it is an almost one to one extension of an equicontinuous system $g: X\to X$. As $\omega(c)$ is a Cantor set, $X$ is not a singleton, hence $g\times g$ is not topologically transitive. On the other hand, $g\times g$ is a factor of $\widehat{f}$, hence topologically transitive. Contradiction!
\end{proof}

{\bf Acknowledgment.} We would like to thank the anonymous referee for his/her valuable comments which led to a revision of this paper.
\section{Nice intervals and children}\label{sec:prelimi}
In this section we shall prove the Main Theorem in the infinitely renormalizable case and deduce it from a Reduced Main Theorem in the non-infinitely-renormalizable case.

Consider a unimodal map $f:[0,1] \to [0,1]$ in $\mathcal{A}_*$.
Let $c$ denote the critical point of $f$ and let $\ell$ be the order of $c$.
Without loss of generality, we may assume $f(0)=f(1)=0$. We will also assume that $f$ is geometrically symmetric near $c$.

\subsection{Notations and terminologies}
Given a subset $V$ of $[0,1]$ and an integer $k\ge 0$, we say that a component $J$  of $f^{-k}(V)$ is a {\em pull back of $V$ by $f^k$}. We say that such a pull back is
\begin{itemize}
\item {\em critical} if it contains the critical point $c$;
\item {\em diffeomorphic} if $f^k$ maps $J$ diffeomorphically onto a component of $V$;
\item {\em unimodal} if $J\ni c$ and $f^{k-1}$ maps a neighborhood of $f(J)$ diffeomorphically onto a component of $V$.
\end{itemize}

For $T\subset [0,1]$, let
$$\dom(T)=\{x\in [0,1]: f^k(x)\in T \mbox{ for some }k\ge 1\}.$$
The {\it first entry map} $R_T: \dom(T)\to T$ is defined as $x\mapsto f^{k(x)}(x)$, where $k(x)$ is the {\it entry time} of $x$ into $T$, i.e., the minimal positive integer such that $f^{k(x)}(x)\in T$. The map $R_T|(\dom(T)\cap T)$ is called the {\it first return map} of $T$. A component of $\dom(T)$ (resp. $\dom(T)\cap T$) is called an {\it entry domain} (resp. {\it return domain}) of $T$. Let
$\mathcal L_x(T)$ denote the entry domain containing $x$.

Let us call an open set $T\subset [0,1]$ {\it nice} if $f^n(\partial T)\cap T=\emptyset$ for all $n\ge 0$ and $T$ does not contain a fixed point of $f$. It is well-known that for such an open set $T$,
\begin{itemize}
\item pull-backs of a nice set are again nice;
\item if $J_j$ is a pull back of $T$ by $f^{k_j}$, $j=1,2$, and $k_1\ge k_2$, then $J_1\cap J_2=\emptyset$ or $J_1\subset J_2$;
\item the entry time is constant in any component of $\dom(T)$, so the first entry map $R_T: D(T)\to T$ is continuous.
\end{itemize}
Moreover, if $f\in\mathcal{A}_*$, then there exists an arbitrarily small symmetric nice interval $T\ni c$. See for example~\cite{M}.

A nice interval $T\ni c$ is called {\em symmetric} if $f(\partial T)$ consists of a single point.
A unimodal pull back of a nice interval $T\ni c$ is also called a {\em child } of $T$.

We say that $f$ is {\em persistently recurrent}
if for each symmetric nice interval $T\ni c$, the number of children of $T$ is finite.
The following is well-known.
\begin{prop}[Blokh-Lyubich~\cite{BL}]\label{prop:persistentrec}
Suppose that $f\in\mathcal{A}_*$ has a Cantor attractor $A$. Then $A=\omega(c)\ni c $, $A$ is a minimal set and $f$ is persistently recurrent.
\end{prop}

Given a bounded interval $I$ and a constant $\tau>0$, let $\tau I$ denote the open interval which is concentric with $I$ and has length $\tau |I|$.  We say that a bounded interval $J$ is $\tau$-well inside an interval $I$ if $I\supset (1+2\tau) J$, i.e., both components of $I\setminus J$ have length at least $\tau |J|$.

A nice interval $I$ is called {\em $\tau$-nice}, if each return domain of $I$ is $\tau$-well inside $I$.

A closed interval $I$ is called a {\em restrictive interval} if $I$ contains $c$ in its interior and there exists an integer $s\ge 2$ such that $I, f(I),\ldots, f^{s-1}(I)$ have pairwise disjoint interior and such that
$f^s(I)\subset I$, $f^s(\partial I)\subset \partial I$. The integer $s$ is called the {\em period } of $I$ and $f^s:I\to I$ is called a {\em renormalization} of $f$. The map $f$ is called {\em infinitely renormalizable} if there exists a restrictive interval with an arbitrarily large period.


\subsection{Nice covers} \label{subsec:nicecover}
Assume that $f\in \mathcal{A}_*$ has a non-periodic recurrent critical point $c$ such that $\omega(c)$ is minimal.
We say that an open cover $\mathcal{Y}$ of $\omega(c)$ is {\em nice} if there is a symmetric nice interval $Y$ such that $\mathcal{Y}$ is the collection of components of $Y\cup \dom(Y)$ which intersect $\omega(c)$. For such an open cover, let
\begin{equation}\label{eqn:qclY}
q(\mathcal{Y}, n)=\#\left\{\text{components of } f^{-n}(Y\cup \dom(Y))\text{ intersecting }\omega(c)\right\}¡£
\end{equation}
For each nice interval $Y\ni c$, let $\nu(Y)$ denote the number of children of $Y$ and for each $n\ge 0$, let $Y_{-n}$ denote the component of $f^{-n}(\dom(Y)\cup Y)$ which contains $c$.
We shall use the following lemma:
\begin{lem}\label{lem:pq} For any symmetric nice interval $Y$ and the corresponding nice cover $\mathcal{Y}$, we have
$$p(\mathcal{Y},n+1)\le q(\mathcal{Y}, n)$$
for each $n\ge 0$. Moreover, for $n$ large enough, we have $$q(\mathcal{Y}, n)\le \sum_{i=0}^{n-1} \nu(Y_{-i}).$$
\end{lem}
\begin{proof} For $x\in\omega(c)$ and $n\ge 0$, if $Z_n(x)$ is the component of $f^{-n}(\dom(Y)\cup Y)$ which contains $x$, then $f^j(Z_n(x))$, $0\le j\le n$, is contained in a component of $\dom(Y)\cup Y$. It follows that $Z_n(x)$ is contained in a element of $\bigvee_{j=0}^{n} f^{-j}\mathcal{Y}$. The first inequality follows. %

Let us prove the second inequality, assume that $n$ is so large that $Y$ has no child with transition time greater than $n$.
For each component $J$ of $f^{-n}(\dom(Y)\cup Y)$, there exists a minimal integer $n'=n'(J)\in \{0,1,\ldots, n\}$ such that $f^{n'}(J)$ contains the critical point $c$.
Let $\mathcal{J}_{n'}$ denote the collection of all components $J$ of $f^{-n}(\dom(Y)\cup Y)$ with $n'(J)=n'$ and $J\cap\omega(c)\not=\emptyset$.
Clearly, $\mathcal{J}_0$ has at most one element. Let us show that $\mathcal{J}_n=\emptyset$. Indeed, any element $J$ is a diffeomorphic pull back of $Y$ by $f^m$ for some $m=m(J)\ge n$, and if $t(J)>1$ is the entry time of $c$ into $J$, then the pull back of $J$ by $f^{t(J)}$ containing $0$ is a child of $Y$ with transition time $\ge m+t(J)>n$, which is ruled out by our assumption on $n$.
A similar argument shows that for each $n>n'>0$,
\begin{equation}\label{eqn:qvschild}
\#\mathcal{J}_{n'}\le \nu(Y_{-n+n'}).
\end{equation}
Indeed, each $J\in\mathcal{J}_{n'}$ is a diffeomorphic pull back of $Y_{-n+n'}$ by $f^{n'}$, so if $t(J)$ is the first entry time of $c$ to $J$,
then the component of $f^{-t(J)}(J)$ which contains $c$ is a child of $Y_{-n+n'}$ with transition time $t(J)+n'$. As different $J$'s correspond to different $t(J)$'s, (\ref{eqn:qvschild}) follows.
Thus
$$q(\mathcal{Y},n)\le \sum_{i=1}^{n-1} \nu(Y_{-i})+1\le \sum_{i=0}^{n-1} \nu(Y_{-i}).$$
\end{proof}
\begin{remark} In Theorem~\ref{thm:lower}, we shall show that $q(\mathcal{Y},n)/n$ is bounded away from zero. However, this does not imply a lower bound for $p(\mathcal{Y},n)$, because an element of $\bigvee_{j=0}^n f^{-j}\mathcal{Y}$ may contain a large number of components of $f^{-n}(\dom(Y)\cup Y)$ intersecting $\omega(c)$.
\end{remark}

The following Reduced Main Theorem is the main step of our proof of the Main Theorem.
\begin{reducedmaintheorem}
Suppose that $f\in\mathcal{A}_*$ is non-renormalizable and that $\omega(c)$ is a Cantor attractor. For each
symmetric nice interval $Y\ni c$, there exists $n_0=n_0(Y)\ge 2$ such that if $T$ is a critical pull back of $Y$ by $f^n$ for some $n\ge n_0$, then the number of children of $T$ is bounded from above by $C\log n$, where $C>0$ is a constant depending only the critical order.
\end{reducedmaintheorem}

\begin{proof}[Proof of the Main Theorem]
We may assume that $f$ is non-renormalizable, as in the infinitely renormalizable case the Main Theorem is well-known, see for exmaple \cite[Proposition III.4.5]{MS}, and the finitely renormalizable case can be reduced to the non-renormalizable case.

Given a nice interval $Y\ni c$, let $\mathcal Y$ denote the corresponding nice cover of $\omega(c)$.
Since $f$ has no wandering interval (\cite[Chapter IV]{MS}), the maximal length of elements of $\mathcal{Y}$ tends to zero as $|Y|\to 0$.
Thus for any open cover $\mathcal{U}$ of $\omega(c)$, there exists a small symmetric nice interval $Y\ni c$ such that
$\mathcal{Y}$ is a refinement of $\mathcal U$, hence $p(\mathcal{U}, n)\le p(\mathcal{Y}, n)$. By Lemma~\ref{lem:pq}, it follows that
$$p(\cU, n+1)\le \sum_{i=0}^{n-1} \nu(Y_{-i})$$
provided that $n$ is large enough.
By Proposition~\ref{prop:persistentrec}, $\nu(Y_{-i})$ is finite for each $i$. By the Reduced Main Theorem, there exists $n_0$ such that for $i\ge n_0$,
$\nu(Y_{-i})\le C\log i$. Thus $p(\cU, n+1)=O(n\log n)$.
\end{proof}

\begin{remark}
In \cite{BV}, it is proved that a Fibonacci-like unimodal map has sub-linear complexity, i.e., $p(\cU, n)\le Cn$ for some constant $C>0$ and each open cover $\cU$.
For a Fibonacci-like unimodal map, the numbers of children of nice intervals are bounded by a constant. Therefore their result is compatible with ours.
\end{remark}

It is not clear to us whether the upper bounds appearing in the Reduced Main Theorem are optimal. Indeed, the following simpler problem is open:

\noindent
{\bf Problem.} {\em Give a positive integer $N\ge 2$, does there exist a real number $\ell_0$ such that if $f\in\mathcal{A}_*$ has critical order $\ell>\ell_0$ and satisfies the following property: each nice interval has at most $N$ children, then $f$ has a wild attractor?}

In \cite[Section 6]{B} Bruin gave a sufficient condition in terms of a different combinatorial language (the kneading map) for existence of wild attractors. Note that Bruin's condition prohibits the existence of saddle-node like returns which however does not seem to be an obstruction for existence of wild attractors.

\subsection{Idea of proof of the Reduced Main Theorem}
We introduce a notion, ``empty space'',  for each small symmetric nice interval, at the beginning of $\S4$. Roughly speaking, we fix a suitable neighborhood $\Lambda$ of $\omega(c)$, and consider the subset $\Lambda(T)$ of $T$ consisting of points which return to $T$ before escaping the neighborhood $\Lambda$. The ``empty space'' $\xi(T)$ measures the proportion of $T\setminus \Lambda(T)$ in $T$: The smaller $\xi(T)$ is, the smaller is the proportion of $T\setminus \Lambda(T)$ in $T$. The assumption that $\omega(c)$ is a Cantor attractor implies that $\xi(T)\to 0$ as $|T|\to 0$.

Most of our effort is to estimate the distortion $\xi(T)$ under unimodal pull back. There are two important principles lying in the proof:
\begin{itemize}
\item If a symmetric nice interval $T$ has many children, then all young children $J$ are $\tau$-nice with a large $\tau$, i.e., all the return domains lie deep inside $J$.
\item If a symmetric  nice interval $T$ is $\tau$-nice and $J$ is a child of $T$, then $\xi(J)/\xi(T)$ is bounded away from zero. Moreover, if $\tau$ is large and $\xi(T)$ is close to zero, then $\xi(J)$ becomes much bigger than $\xi(T)$.
\end{itemize}

The proof of the Reduced Main Theorem occupies the next three sections.
In \S\ref{sec:pre}, we study the size of children of a given nice interval and the geometry of their return domains. In \S\ref{sec:emptyspace}, we study the distortion of ``empty space'' under pull backs.
In both cases, the presence of central cascade is an unpleasant situation and responsible for most complications of the arguments. The proof of the Reduced Main Theorem is completed in \S\ref{sec:reducedmainthm}.

\section{Real bounds}\label{sec:pre}
Consider a map $f\in\mathcal{A}_*$ with a recurrent critical point $c$. We say a constant is {\em universal} if it depends only on $\ell$. In this section, we shall obtain upper bounds of length of children of given nice intervals and the geometry of their return domains. The main result is Proposition~\ref{prop:sizeofchildren}.

\subsection{Preliminaries}

The Koebe principle is the main tool to control distortion in one-dimensional dynamics. The following version was taken from \cite[Proposition 1]{BRSS}, whose proof is based on previous results in the literature, in particular \cite[Theorem C]{SV}.
\begin{thm}\label{thm:koebe}
There exists $\eta(f)>0$ such that the following holds. Let $s\ge 1$ be an integer and let $T$ be an interval. Assume that $f^s|T$ is a diffeomorphism onto its image and that $|f^s(T)|<\eta(f)$. If $J$ is a subinterval of $T$ such that $f^s(J)$ is $\tau$-well inside $f^s(T)$, then
\begin{itemize}
\item [1.] for any $x,y\in J$, $$0.9\left(\dfrac{\tau}{1+\tau}\right)^2\le \dfrac{|Df^s(x)|}{|Df^s(y)|}\le \dfrac{1}{0.9}\left(\dfrac{1+\tau}{\tau}\right)^2;$$
\item [2.] $J$ is $\tau'$-well inside $T$, where $\tau'=\dfrac{0.9\tau^2}{1+2\tau}$.
\end{itemize}
\end{thm}

Given a symmetric nice interval $I\ni c$, we shall use the following notation: $I^0=I$ and $I^{k+1}$ is the return domain of $I^k$ that contains $c$.
The sequence
$$I^0\supset I^1 \supset I^2 \supset \cdots,$$
is often called the {\it principal nest} starting from $I$.
The first return map $R_{I^n}:I^{n+1} \to I^n$ is called {\it central} if $R_{I^n}(c)\in I^{n+1}$ and {\it non-central} otherwise. We say that $R_{I^n}: I^{n+1}\to I^n$ is {\em high} if
$R_{I^n}(I^{n+1})\ni c$ and {\em low} otherwise.

The following Real Bounds theorem was first proved by Martens~\cite{M} in the case that $f$ has negative Schwarzian derivative, and extended to general smooth unimodal maps in~\cite{K}.

\begin{thm}\label{thm:realbounds}
There exists a universal constant $\rho>0$ such that for any small symmetric nice interval $I^0\ni c$,
the following hold:
\begin{enumerate}
\item[(i)] If $R_{I^0}: I^1\to I^0$ is non-central and low, then $I^1$ is $\rho$-well inside $I^0$;
\item[(ii)] If $R_{I^0}: I^1\to I^0$ is non-central and high, then $I^2$ is $\rho$-well inside $I^1$;
\item[(iii)] If $I^1$ is not $\rho$-well inside $I^0$, then $f^{s-1}$ maps a neighborhood of $f(I^1)$
diffeomorphically onto a $\rho$-scaled neighborhood of $I^0$, where $s$ is the return time of $c$ into $I^0$. In particular,
the map $f^{s-1}|f(I^1)$ has uniformly bounded distortion: for any $x, y\in I^1$,
$$|Df^{s-1}(f(x))|\le K(\rho) |Df^{s-1}(f(y))|,$$
where $K(\rho)>1$ is a constant.
\end{enumerate}
\end{thm}

A sequence of open intervals $\{T_j\}_{j=0}^s$ is called a {\it chain} if for each $j=0,1,\ldots, s-1$, $T_j$ is a component of $f^{-1}(T_{j+1})$. The {\it order} of the chain is the number of $j$'s with $0\le j <s$ such that $T_j$ contains the critical point $c$.

The following theorem is an improvement of \cite[Theorem C(1)]{SV} for unimodal maps, which gives relationship between the constants $\tau$ and $\tau'$.
\begin{thm}\label{thm:SV}
Assume that $f$ is not infinitely renormalizable. For any $\tau>0$ there exists $\tau'>0$, such that the following holds. Let $c\in J\subset I$ be small symmetric nice intervals such that $J$ is $\tau$-well inside $I$. Then for any $x\in \dom(J)$, $\mathcal L_x(J)$ is $\tau'$-well inside $\mathcal L_x(I)$. Moreover, for each constant $\tau_*>0$ there exist constants $C=C(\tau_*)>0, \alpha=\alpha(\tau_*)>0$ such that if $\tau>\tau_*$, then we can choose $\tau'$ such that
\begin{equation}\label{eqn:tau'quantified}
\tau'\ge C\tau^\alpha.
\end{equation}
\end{thm}
\begin{proof}
By Theorem~\ref{thm:koebe} and non-flatness of the critical point, it suffices to prove the statement for $x\in \dom(J)\setminus I$. Let $I^0=I$. Let $m(0)=0$ and $1\le m(1)<m(2)<\cdots$ be all the non-central return moments, i.e., the return map $R_{I^{m(k)-1}}$ is non-central. Since $f$ is not infinitely renormalizable, $|I^n|\to 0$ as $n\to\infty$, provided that $I$ is small enough.
So there exists $k\ge 0$  such that
$$I^0\supset I^{m(1)}\supset \cdots \supset I^{m(k)}\supsetneq J\subset  I^{m(k+1)}.$$

Define $\tau_i$, $1\le i\le k+1$ such that
$$\frac{|I^{m(i-1)}|}{|I^{m(i)}|}:=1+2\tau_i, \mbox{ for } 1\le i\le k,$$
and $$\frac{|I^{m(k)}|}{|J|}:= 1+ 2\tau_{k+1}.$$
Then
\begin{equation}\label{eqn:tauproduct}
1+2\tau =\frac{|I|}{|J|}=\prod_{i=1}^{k}\frac{|I^{m(i-1)}|}{|I^{m(i)}|}\cdot\dfrac{|I^{m(k)}|}{|J|}=\prod_{i=1}^{k+1} (1+2\tau_i).
\end{equation}

For each $1\le i\le k$, the first entry map $R_{I^{m(i)}}:\mathcal L_x(I^{m(i)})\to I^{m(i)}$ can be extended diffeomorphically onto $I^{m(i-1)}$ (see Lemma~\ref{lem:Ffull}). By Theorem~\ref{thm:koebe}, $\mathcal L_x(I^{m(i)})$ is $\tau_i'$-well inside $\mathcal L_x(I^{m(k-1)})$, where $\tau_i'=0.9\frac{\tau_i^2}{1+2\tau_i}$. Similarly, since $I^{m(k)}\supset J\supset I^{m(k+1)}$, the first entry map $R_J:\mathcal L_x(J) \to J$ can be extended diffeomorphically onto $I^{m(k)}$, and $\mathcal L_x(J)$ is $\tau_{k+1}'$-well inside $\mathcal L_x(I^{m(k)})$, where $\tau_{k+1}'=0.9\frac{\tau_{k+1}^2}{1+2\tau_{k+1}}$. In conclusion, $\mathcal{L}_x(J)$ is $\tau'$-well inside $\mathcal{L}_x(I)$, where
\begin{equation}\label{eqn:esttau'}
1+2\tau'=\prod_{i=1}^{k+1} (1+2\tau_i').
\end{equation}

Let us prove that $\tau'$ is bounded away from zero. By Theorem~\ref{thm:realbounds}, for each $2\le i\le k$, we have $\tau_i\ge \rho$. So we are done if $k\ge 2$. If $k\le 1$, then by (\ref{eqn:tauproduct}), $(1+2\tau_i)^2\ge 1+2\tau$ holds for $i=1$ or $2$, thus $\tau'$ is bounded from below by a positive constant depending on $\tau$.

Now assume $\tau$ is bounded from below by a constant $\tau_*>0$ and let us prove (\ref{eqn:tau'quantified}). Let
$\rho_*=\min (\tau_*, \rho)$, and let
$$\mathcal{I}=\{1\le i\le k+1: (1+2\tau_i)^4 > 1+ 2\rho_*\}.$$
Then $\{2,\ldots, k\}\subset \mathcal{I}$.  So by (\ref{eqn:tauproduct}),
\begin{equation}\label{eqn:calIprod}
\prod_{i\in\mathcal{I}} (1+2\tau_i)\ge \frac{1+2\tau}{\sqrt{1+2\rho_*}}\ge \sqrt{1+2\tau}.
\end{equation}
For each $i\in\mathcal{I}$, $\tau_i'$ is bounded away from zero, so there exists a constant $\mu\in (0,1)$ such that
$\tau_i'\ge \mu \tau_i$. Thus by (\ref{eqn:esttau'}),
$$1+2\tau'\ge \prod_{i\in\mathcal{I}} (1+2\mu \tau_i)\ge \prod_{i\in\mathcal{I}} (1+2\tau_i)^\mu.$$
Together with (\ref{eqn:calIprod}), this implies
$$1+2\tau'\ge (1+2\tau)^{\mu/2}.$$
The inequality (\ref{eqn:tau'quantified}) follows.
\end{proof}

Recall that a child $J\ni c$ of a symmetric nice $I\ni c$ is a unimodal pull back of $I$ by $f^s$ for some $s\ge 1$. The integer $s$ is called a {\em transition time} from $J$ to $I$.

\begin{lem}\label{lem:childtime}
Let $J$ be a child of $I$ with transition time $s$, then for each $x\in J$, the return time of $x$ to $J$ is not less than $s$.
\end{lem}

\begin{proof}
Let $\{J_i\}_{i=0}^s$ be the chain with $J_0=J$ and $J_s=I$. Since $f^{s-1}:J_1\to J_n$ is diffeomorphic, $c \notin J_i$ for each $1\le i \le s-1$. Therefore $J_i\cap J=\emptyset \, (1\le i \le s-1)$, since otherwise $J_i\supset J\ni c$, which is impossible. For each $x\in J$, $f^i(x)\in J_i \,(1\le i \le s-1)$. If $f^k(x)\in J$, then $k\ge s$.
\end{proof}

\begin{lem}\label{lem:wellinsidenice}
Let $I\ni c$ be a small nice interval and let $J$ be a child of $I$. Assume that $J$ is $\tau$-well inside $I$. Then $J$ is a $\tau'$-nice interval, where $\tau'>0$ depends only on $\tau$. Moreover, when $\tau$ is sufficiently large, we have $\tau'\ge C\tau^\alpha$ for some constants $C>0,\alpha>0$.
\end{lem}
\begin{proof} Let $s$ be the transition time of $J$ into $I$.
Take an arbitrary $x\in D(J)\cap J$ and let $r$ be the first return time of $x$ into $J$. By Lemma~\ref{lem:childtime}, $r\ge s$.
Let $U:=f^s(\mathcal{L}_x(J))$. By Theorem~\ref{thm:SV}, $\mathcal{L}_{f^s(x)}(J)$ is $\tau_1$-well inside $I$, where $\tau_1>0$ is a constant depending only on $\tau$, and when $\tau>1$, there exist constants $C_1>0$ and $\alpha_1>1$ such that $\tau_1\ge C_1\tau^{\alpha_1}$.
Since $f^{s-1}$ maps a neighborhood of $f(J)$ diffeomorphically onto $I$, by the Koebe principle and non-flatness of the critical point, $\mathcal{L}_x(J)$ is $\tau'$-well inside $J$, for some constant $\tau'>0$ depending only on $\tau_1$. Moreover when $\tau$ is sufficiently large, $\tau_1>1$, and  we can choose $\tau'=C_0\tau_1^{1/\ell}$, where $C_0$ is a constant.   Thus the lemma holds with $C=C_0C_1^{1/\ell}$ an $\alpha=\alpha_1/\ell$.
\end{proof}

\begin{lem} \label{lem: boundsofchild}
There exists a universal constant $\rho_0>0$ such that if $I\ni c$ is a small nice interval and $J\ne I^1$ is a child of $I$, then $J$ is $\rho_0$-well inside $I$.
\end{lem}
\begin{proof}
Let $s$ be the return time of $c$ to $I$ and let $m\ge 1$ be such that $f^{s}(c)\in I^{m-1}\setminus I^m$. Note that $J\subset I^m$ and $f^s(J)\subset I^{m-1}\setminus I^m$.

Let $\rho>0$ be the constant appearing in Theorem~\ref{thm:realbounds}. If $I^m$ is $\rho$-well inside $I^{m-1}$, then $J$ is $\rho$-well inside $I$, and we are done. So assume that $I^m$ is not $\rho$-well inside $I^{m-1}$.
Then $R_{I^{m-1}}: I^m \to I^{m-1}$ is a high return, and $f^{s-1}|f(I^m)$ has uniformly bounded distortion. Since $f^s(I^m)$ is definitely larger than $f^s(J)$, it follows that $|f(J)|/|f(I^m)|$ is bounded away from one, hence $J$ is uniformly well inside $I^m\subset I$.
\end{proof}

\subsection{Central cascade}\label{sec:sizeofchildren}
By a {\it central cascade}, we mean a sequence of symmetric nice intervals
$$T\supset T^1\supset \cdots T^m, \quad (m\ge 1)$$
which contain $c$ such that
\begin{itemize}
\item $T^{i+1}$ is the central return domain of $T^i$, for each $0\le i<m$;
\item  the first return times of $c$ to $T,T^{1}, \cdots, T^{m-1}$ are all the same.
\end{itemize}
So $R_{T^i}$ are central for all $0\le i \le m-2$.
A central cascade is called {\em maximal} if $R_{T}(c)\not\in T^{m}$.


\begin{prop}\label{prop:sizeofchildren}
Let $T=T^0\ni c$ be a small symmetric nice interval and let $T^0\supset T^1\supset T^2\supset\cdots\supset T^m$ be a maximal central cascade.
Assume that $T^1$ is $\tau$-well inside $T^0$. Let $i\in \{1,2,\ldots, m\}$ and
let $J_1\supsetneq J_2\supsetneq\cdots$ be all the children of $T^i$.
Then there exist constants $C>0$ and $0<\lambda< \lambda_0<1$, depending only on $\tau$, such that
\begin{enumerate}
\item[1.] for each $k=1,2,\ldots$, we have $|J_k|\le \lambda^{k-1}|T^i|;$
\item[2.] for each $k\ge 2$, $J_k$ is $C\lambda_0^{-k}$-nice.
\end{enumerate}
\end{prop}
To prove this proposition, let us first introduce some notation. For $y\in D(T^0)$,
let $r(y)$ denote the first entry time of $y$ into $T^0$, and let $s=r(c)$, so $R_{T^0}|T^1= f^s|T^1$.
Let
$$E(T)=\bigcup_{i=0}^{m-1}\{x\in T^i\setminus T^{i+1}: R_T^{i}(x) \in \dom (T) \},$$
and for each $x\in E(T)\cap (T^i\setminus T^{i+1})$, let
$$t(x)=is + r(f^{is}(x)).$$
Moreover, let $F=F_{T}: E(T)\to T$ be defined as
$$F(x)= f^{t(x)}(x).$$
Clearly, $t(x)$ is constant on each component $J$ of $E(T)$.

We shall also need the following notations:
\begin{itemize}
\item $Q=f^{-s}(T^m)\cap T^m$;
\item $V$ is the component of $f^{-s}(E(T))$ which contains $c$;
\item $X=f^{-s}(E(T))\cap (T^m\setminus (Q\cup V))$.
\end{itemize}

\begin{lem}\label{lem:Ffull}
\begin{enumerate}
\item[(i)] The map $F$ maps each component $J$ of $E(T)$ diffeomorphically onto $T$.
\item[(ii)] For each $x\in D(T^m)\setminus (Q\cup V)$, if $k$ is the entry time of $x$ to $T^m$, then $f^k$ maps a neighborhood $W(x)$ of $\mathcal{L}_x(T^m)$ diffeomorphically onto $T$. Moreover, if, in addition, $x\in X$ then $W(x)\subset X$.
\end{enumerate}
\end{lem}
\begin{proof}
We first prove the statement (i). If $J$ is a component of $E(T)$ in $T^0\setminus T^1$, then $J$ is a non-central return domain and $F=R_T$, so $F$ maps $J$ diffeomorphically onto $T$.
Now let $J$ be a component of $E(T)$ in $T^i\setminus T^{i+1}$ for some $1\le i<m$. Since $f^{is}$ maps a component of $T^i\setminus T^{i+1}$ diffeomorphically onto a component of $T^0\setminus T^1$, $f^{is}: J\to J':=f^{is}(J)$ is a diffeomorphism and $J'$ is a component of $E(T)$ in $T^0\setminus T^1$. Since $t|J= t|J'+is$, $F|J= R_T|J'\circ f^{is}|J$ maps $J$ diffeomorphically onto $T$.

Let us prove the statement (ii). Let us distinguish a few cases.

{\bf Case 1.} $x\in T\setminus T^m$. In this case, $f^k|\mathcal{L}_x(T^m)$ can be written as an iterate of $F$, so the statement follows from (i). Note that $W(x)\subset E(T)$.

{\bf Case 2.} $x\not\in T$. Let $k'\le k$ be the first entry time of $x$ to $T$. Then $f^{k'}:\mathcal{L}_x(T)\to T$ is a diffeomorphism. So the statement holds if $k'=k$. If $k'<k$, then $f^{k'}(\mathcal{L}_x(T^m))=\mathcal{L}_{f^{k'}(x)}(T^m)$ and we are reduced to Case 1.

{\bf Case 3.} $x\in D(T^m)\cap X$. Then $k>s$ and $x'=f^s(x)\in D(T^m)\cap (T^{m-1}\setminus T^m)$.
Let $W_0(x)$ and $W_0(x')$ denote the component of $X$ which contains $x$ and $x'$ respectively. By definition of $X$, $f^s:W_0(x)\to W_0(x')$ is a diffeomorphism.
So we are reduced to Case 1 again.
\end{proof}

A nice interval $I$ is called {\em $\tau$-non-central nice} if all its return domains, except possibly the one containing $c$, are $\tau$-well inside $I$.
The following is an immediate consequence of Lemma~\ref{lem:Ffull}.
\begin{lem} \label{lem:noncentralnice}
Assume that $T^1$ is $\tau$-well inside $T^0$. Then for each $1\le i<m$, $T^i$ is a $\tau'$-non-central-nice interval, where $\tau'$ depends only on $\tau$.
\end{lem}
\begin{proof} Note that for each return domain $U$ of $T^i$, $U\not= T^{i+1}$, the first return map $R_{T^i}|U$ can be written in the form $F^n|U$ for some $n\ge 1$. By Lemma~\ref{lem:Ffull}, it follows that
$f^k: U\to T^i$ extends to a diffeomorphism $f^k: \hU\to T^0$ and $\hU\subset T^i\setminus T^{i+1}$, where $k\ge 1$ is the first return time of $J$ into $T^i$. Since $T^i\subset T^1$ is $\tau$-well inside $T$, by the Koebe principle, $U$ is well inside $\hU$, hence well-inside $T^i$.
\end{proof}

\begin{lem}\label{lem:largebounds}
For any $\tau>0$, there exists $\tau'>0$ such that if $T\ni c$ is a small $\tau$-non-central-nice interval and $J_2\subset J_1$ are children of $T$, then $J_2$ is $\tau'$-well inside $J_1$.
\end{lem}

\begin{proof}
 Let $s_1<s_2$ be the transition time of $J_1,J_2$ to $T$ respectively. Let $s$ be the maximal integer such that $s_1\le s<s_2$ and $x:=f^{s}(c)\in T$. Then $s_2-s$ is the return time of $x$ into $T$ and $\mathcal{L}_x(T)$ does not contains $c$. By assumption, $\mathcal{L}_x(T)$ is $\tau$-well inside $T$, so by Theorem~\ref{thm:SV}, the component of $f^{-(s-s_1)}(\mathcal{L}_x(T))$ containing $f^{s_1}(c)$ is $\tau'$-well inside $T$, where $\tau'>0$ is a constant. By Theorem~\ref{thm:koebe} and non-flatness of the critical point, $J_2$ is well inside $J_1$.
\end{proof}

\newcommand{\hJ}{\widehat{J}}

These lemmas imply Proposition~\ref{prop:sizeofchildren} immediately unless
\begin{equation}\label{eqn:notwellinside}
(1+2\rho)T^m\supset T^{m-1}.
\end{equation}
To deal with the case when (\ref{eqn:notwellinside}) holds, we need the following three lemmas.

Assume (\ref{eqn:notwellinside}). Then by Theorem~\ref{thm:realbounds}, $R_{T^{m-1}}: T^m\to T^{m-1}$ is high, so $Q$ consists of two intervals, each of which is mapped diffeomorphically onto $T^m$ by $f^s$. Let $Q_+, Q_-$ denote the components of $Q$ such that $f^s|Q_+$ is monotone increasing.
Let $b$ be the unique fixed point of $f^s|Q_-$, let $\hat{b}=(f^s|Q_+)^{-1}(b)$ and let $B=(b, \hat{b})$.

\newcommand{\hZ}{\widehat{Z}}
\begin{lem}\label{eqn:bddisfstm}
There exist universal constants $K>1$ and $\sigma>0$  such that the following hold:
\begin{enumerate}
\item [(i)] For any $x, y\in T^m$, $|Df^{s-1}(f(x))|\le K |Df^{s-1} (f(y))|$;
\item [(ii)] $|(f^s)'(x)|\le K$ holds for all $x\in T^m$;
\item [(iii)] for any measurable $A\subset T^m$, $\frac{|A|}{|T^m|}\le K \left(\frac{|f^s(A)|}{|T^m|}\right)^{1/\ell}$;
\item [(iv)] $f^s$ maps a neighborhood $Z$ of $Q_+$ diffeomorphically onto its image and $\hZ:=f^s(Z)\supset Z\cup T^m\cup (1+2\sigma) Q_+$.
\item[(v)]
$\sigma |T^m| \le |B|\le (1-\sigma)|T^m|.$
\end{enumerate}
\end{lem}
\begin{proof}
By Theorem~\ref{thm:realbounds}, $f^{s-1}$ maps a neighborhood $G_1$ of $f(T^m)$ diffeomorphically onto $G:=(1+2\rho) T^{m-1}$.
By the real Koebe principle, there exists $K>1$ such that (i) holds.
For $x\in T^m$, we have
$$|(f^{s})'(x)|=|f'(x)||(f^{s-1})'(f(x))|\le K\frac{|f^s(T^m)|}{|f(T^m)|}|f'(x)|.$$ Since
$|f^s(T^m)|\le |T^{m-1}|\le (1+2\rho) |T^m|$, by the non-flatness, it follows that the statement (ii) holds by redefining the constant $K$.
The statement (iii) follows from (i) in a similar way.
For (iv) and (v), assume for definiteness that $Q_+$ lies to the left of $c$. Let $G_0=f^{-1}(G_1)$ and let $Z=(u,c)$ be the left component of $G_0\setminus \{c\}$.  Then $f^s$ maps $Z$ diffeomorphically onto its image and $\hZ:=f^s(Z)\supset T^m$. Since $f^s(u)$ is the left endpoint of $G$, we have $f^s(Z)\supset (1+2\sigma) Q_+$, where $\sigma=\min (1, \rho)$. If $Z\not\subset f^s(Z)$, then $f^s$ would map $Z\setminus I^m$ into itself and hence $f^s$ would have an attracting fixed point, which is not possible. Thus $f^s(Z)\supset Z$. The statement (iv) is proved.
The statement (iv) follows from (ii), since $f^s(Q_-\cap \overline{B})\supset Q_-\setminus B$ and
$f^s(Q_-\setminus B)\supset Q_-\cap B$.
\end{proof}

\begin{lem}\label{lem:setQ}
Assume that $(1+2\rho) T^m\not\subset T^{m-1}$. Then there exists a universal constant $\theta\in (0,1)$ such that if $P$ is an interval such that $f^{js}(P)\subset Q$ for $j=0,1,\ldots, N-1$, then
$$|P|\le \theta^N |T^m|.$$
\end{lem}
\begin{proof}
Let
$$\mathcal{P}_n=\{x\in T^m: f^{is}(x)\in Q\mbox{ for } 0\le i<n\}$$
and
$$\mathcal{P}_n^*=\{x\in \mathcal{P}_n: f^{ns}(x)\in T^m\setminus \{b, \hat{b}\}\}.$$
Note that each component of $\mathcal{P}_n$ is the union of three intervals of $\mathcal{P}_n^*$, up to two points (corresponding to preimages of $b$ and $\hat{b}$). As each component of $\mathcal{P}_n$, $n\ge 1$, is at most of length $|T^m|/2$, it suffices to show there exist universal constants $C_*>0$ and $\theta_*\in (0,1)$ such that
for each component $P_n^*$ of $\mathcal{P}_n^*$ we have
\begin{equation}\label{eqn:pn*}
|P_n^*|\le C_* \theta_*^n |T^m|.
\end{equation}

Let $Q_{+,1}= Q_+$ and $Q_{-,1}=Q_-$, and for each $j>1$, let $Q_{+,j}= (f^s|Q_+)^{-1}(Q_{+,j-1})$ and $Q_{-,j}=(f^s|Q_-)^{-1}(Q_{+,j-1})$. Then $Q_{+,j}$ are symmetric to $Q_{-, j}$ with respect to $c$.

{\bf Claim 1. } There exists a universal constant $\theta_1\in (0,1)$ such that
\begin{equation*}
|Q_{+,j}|=|Q_{-,j}|\le \theta_1^j |T^m|.
\end{equation*}
Indeed, by (iv) of Lemma~\ref{eqn:bddisfstm}, for each $j\ge 1$, $f^{js}$ maps a neighborhood of  $Q_{+,j}$ diffeomorphically onto $\hZ$.  Since $f^{js}(Q_{+,j+1})=Q_+$, and $f^{js}(Q_{+,j})=T^m$,
it follows by the Koebe principle that $|Q_{+,j+1}|/|Q_{+,j}|$ is uniformly bounded away from $1$. The claim follows.

Let
$$\mathcal{B}_n=\{x\in T^m: f^{js}(x)\in Q\mbox{ for } 0\le j<n, f^{ns}(x)\in B\}\subset \mathcal{P}_n^*.$$
For each component $B_n$ of $\mathcal{B}_n$, $f^{ns}$ maps a neighborhood of $B_n$ diffeomorphically onto $T^m$. By (v) of Lemma~\ref{eqn:bddisfstm}, $B$ is uniformly well inside $T^m$. By the Koebe principle, there exists a universal constant $K_1>1$ such that
\begin{equation}\label{eqn:dist2b}
\sup_{x,y\in B_n}\frac{|(f^{ns})'(x)|}{|(f^{ns})'(y)|}\le K_1.
\end{equation}

{\bf Claim 2.} There exists a universal constant $\theta_2\in (0,1)$ such that for each component $B_n$ of $\mathcal{B}_n$, $n=0,1,\ldots$, we have
\begin{equation}\label{eqn:bn}
|B_n|\le \theta_2^n|T^m|.
\end{equation}

To prove this claim, let $\mathcal{B}=\bigcup_{n=1}^\infty \mathcal{B}_n\subset D(B)$. For each $x\in \mathcal{B}$, the first entry time of $x$ into $B$ is of the form $k(x) s$, where $k(x)\ge 1$ is an integer.
For $x\in \mathcal{B}\setminus B$, we have $f^{js}(x)\in Q_+$ for $1\le j<k(x)$, so $\mathcal{L}_x(B)\subset Q_{+, k(x)}\cup Q_{-, k(x)}$. Thus by Claim 1, we have
\begin{equation}\label{eqn:mathcallxB}
|\mathcal{L}_x(B)|\le \theta_1^{k(x)} |T^m|\mbox{ holds for all }x\in\mathcal{B}\setminus B.
\end{equation}
Let us now show that there exist a universal constant $\theta_3\in (0,1)$ such that
\begin{equation}\label{eqn:mathcallxB1}
|\mathcal{L}_x(B)|\le \theta_3^{k(x)} |B|\mbox{ holds for all }x\in\mathcal{B}\cap B.
\end{equation}
Indeed, $\mathcal{L}_x(B)$ lies in a component of $B\setminus\{c\}$, so
\begin{equation}\label{eqn:lxbhalf}
|\mathcal{L}_x(B)|\le |B|/2.
\end{equation}
In particular, if $k(x)=1$, then (\ref{eqn:mathcallxB1}) holds with $\theta_3=1/2$. If $k(x)>1$, then $f^s(x)\in \mathcal{B}\setminus B$ and $f^{s}(\mathcal{L}_x(B))=\mathcal{L}_{f^s(x)}(B)$. So by (\ref{eqn:mathcallxB}) and part (iii) and (v) of Lemma~\ref{eqn:bddisfstm}, we have
$$\frac{|\mathcal{L}_x(B)|}{|B|}\le \sigma_2^{-1} \frac{|\mathcal{L}_x(B)|}{|T^m|}\le \frac{K_2}{\sigma_2} \theta_1^{(k(x)-1)/\ell}.$$
Together with (\ref{eqn:lxbhalf}), it follows that
(\ref{eqn:mathcallxB1}) holds for a suitable choice of $\theta_3$.

Now let us prove (\ref{eqn:bn}). Take a component $B_n$ of $\mathcal{B}_n$, $n\ge 1$.  Let $1\le n_1< n_2<\cdots < n_k=n$ be all the positive integers such that $f^{n_is}(B_n)\subset B$ and let $B_{n_i}$ be the component of $\mathcal{B}_{n_i}$ which contains $B_n$.
Then $B_{n_1}\supset B_{n_2}\supset\cdots\supset B_{n_k}=B_n$.
For each $1\le i<k$,
$Y_i:=f^{n_{i}s}(B_{n_{i+1}})$ is a component of $\mathcal{B}$ with entry time $(n_{i+1}-n_i)s$.
By (\ref{eqn:dist2b}),
$$\frac{|B_{n_{i+1}}|}{|B_{n_i}|}\le \frac{K_1 |Y_i|}{|B|+(K_1-1)|Y_i|}.$$
Thus by (\ref{eqn:mathcallxB1}), there exists $\theta_4\in (0,1)$ such that
$|B_{n_{i+1}}|\le \theta_4^{n_{i+1}-n_i} |B_{n_i}|.$
So $|B_n|\le \theta_4^{n-n_1}|B_{n_1}|$.
Let $\theta_2=\max(\theta_1,\theta_3,\theta_4)$.
By (\ref{eqn:mathcallxB}) and (\ref{eqn:mathcallxB1}), $|B_{n_1}|\le \theta_2^{n_1}|T^m|$.
Thus (\ref{eqn:bn}) holds.

Now let us complete the proof. Let $P_n^*$ be a component of $\mathcal{P}_n^*$.
We may assume $P_n^*\not\in\mathcal{B}_n$ for otherwise, Claim 2 applies.
Write $P_i^*=f^{(n-i)s} (P_n^*)$. Let $n_0$ be maximal in $\{0,1,\ldots, n\}$ such that $P_{i}^*\cap B=\emptyset$ for all $0\le i\le n_0$.
Since $f^s(Q\setminus B) \cap (Q_-\setminus B)=\emptyset$, we have $f^{is}(P_{n_0}^*)\subset Q_+$ for all $1\le i<n_0$. So
$P_{n_0}^*\subset Q_{+, n_0}$ or $P_{n_0}^*\subset Q_{-, n_0}$.
By Claim 1,
we have $|P_{n_0}^*|\le \theta_1^{n_0} |T^m|$.
If $n_0=n$ then we are done again. Assume $n_0<n$. Then $P_{n_0+1}^*\subset B$.
By part (iii) and (v) of Lemma~\ref{eqn:bddisfstm},
$|P_{n_0+1}^*|\le C_1'\theta_1'^{n_0+1}|B|$ holds for some universal constants $C_1'>0$ and $\theta_1'\in (0,1)$.
Let $B_{n-n_0-1}$ be the component of $\mathcal{B}_{n-n_0-1}$ which contains $P_n^*$. By (\ref{eqn:dist2b}), we have
$$|P_n^*|\le K_1 |B_{n-n_0-1}| \frac{|P_{n_0+1}^*|}{|B|}.$$
By Claim 2, the inequality (\ref{eqn:pn*}) follows.
\end{proof}

\begin{lem}\label{lem:fromV}
Assume $(1+2\rho)T^m\not\subset T^{m-1}$. Let $y\in V$ and let $t$ be the first return time of $y$ to $T^m$. Assume that
$f^{js} (f^t(y))\in Q$ for all $j=0,1,\ldots, n-1$ and let $H$ be the component of $f^{-ns-t}(T^m)$ which contains $y$. Then
$$|H|\le \theta_0^n |T^m|,$$
where $\theta_0\in (0,1)$ is a constant depending on $\tau$.
\end{lem}
\begin{proof} Let $\delta\in (0,\rho)$ be such that $|T^{m-1}|=(1+2\delta)|T^m|.$ Since $f^s(V)\subset T^{m-1}\setminus T^m$ we have $|f^s(V)|\le \delta |T^m|$. By part (i) of Lemma~\ref{eqn:bddisfstm},
$|fV|\le K\delta |f(T^m\setminus V)|.$
By non-flatness, there exist universal constants $K_1>1$ and $\eta_1\in (0,1)$ such that
$$\frac{|V|}{|T^m|}\le \min (\eta_1, K_1\delta^{1/\ell})=:\eta.$$
Since $H\subset V$, we obtain
\begin{equation}\label{eqn:lengthV}
|H|\le \eta |T^m|.
\end{equation}
Take $\gamma\in (0,1)$ such that $K_1^\gamma \eta_1^{1-\gamma}=1$.

\noindent
{\em Case 1.} $\delta< \theta^{n/2}$. Then
$$\eta\le (K_1\delta^{1/\ell})^\gamma \eta_1^{1-\gamma}\le \theta^{n \gamma/(2\ell)},$$
so we are done in this case.

\noindent
{\em Case 2.} $\delta \ge \theta^{n/2}$. By Lemma~\ref{lem:setQ}, $|f^t(H)|\le\theta^n |T^m|$.
So $f^t(H)$ is $\theta^{-n/2}$-well inside $T^{m-1}$. By Lemma~\ref{lem:Ffull} (ii), $f^{t-1}$ maps an interval $W\ni f(y)$ diffeomorphically onto $T^{m-1}$. Let $W_0$ be the component of $f^{-1}(W)$ which contains $y$. Then $W_0\subset T^m$.  By the Koebe principle and non-flatness, we obtain that
$$|H|\le C \theta^{n/2\ell}|W_0|\le C\theta^{n/2\ell} |T^m|,$$ where $C=C(\theta)$ is a constant.
Together with (\ref{eqn:lengthV}), this implies the statement.
\end{proof}

\begin{proof}[Proof of Proposition~\ref{prop:sizeofchildren}]
The second statement follows from the first by Lemma~\ref{lem:wellinsidenice}. In the following we shall prove the first statement.

By Lemmas~\ref{lem:noncentralnice} and~\ref{lem:largebounds}, the first statement holds in the case $1\le i<m$. In the following, we shall estimate the size of children of $T^m$.

If $(1+2\rho) T^m\subset T^{m-1}$, then by Lemma~\ref{lem:wellinsidenice}, $T^m$ is $\rho'$-nice for some $\rho'>0$ and so we are done again by Lemma~\ref{lem:largebounds}. We assume from now on that $(1+2\rho) T^m\not\subset T^{m-1}$, so that Lemmas~\ref{eqn:bddisfstm}, ~\ref{lem:setQ} and~\ref{lem:fromV} apply.

For each $i=1,2,\ldots$, let $S_i$ denote a transition time from $J_i$ to $T^m$. By definition, $f^{S_i-1}$ maps an interval
$\hJ_i$ which contains $f(J_i)$ diffeomorphically onto $T^m$.
Let $i(1)=\inf\{i\ge 1: f^{S_i}(c)\not\in Q\},$ and define inductively,
$$i(j+1)=\inf\{i> i(j): f^{S_{i}}(c)\not\in Q\}.$$
For $i\in \{1,2,\ldots, i(1)\}$, applying Lemma~\ref{lem:fromV} to $y=c$ and $n=i$, we obtain that
$|J_{i}|\le \theta_0^{i}|T^m|$.

It remains to show that for each $i(j)<i\le i(j+1)$,
$|J_{i}|\le \theta_2^{i-i(j)} |J_{i(j)}|$ holds for some constant $\theta_2=\theta_2(\tau)\in (0,1)$.
To this end, let $y:=f^{S_{i(j)}}(c)$ and we distinguish two cases.

{\em Case 1.} $y\in X$. Then $k:=S_{i(j)+1}-S_{i(j)}$ is the first return time of $y$ into $T^m$.
By Lemma~\ref{lem:Ffull}, $f^k$ maps an interval $W(y)$ with $y\in W(y)\subset X$ diffeomorphically onto $T\supset (1+2\tau)T^m$.
By Lemma~\ref{lem:setQ}, $|f^{S_{i(j)+1}}(J_i)|\le \theta^{i-i(j)-1}|T^m|$, so $f^{S_{i(j)+1}}(J_i)$ is
is $\tau \theta^{i(j)-i+1}$-well inside $T$.
By the Koebe principle, $f^{S_{i(j)}}(J_i)$ is $\tau'\theta'^{i(j)-i}$-well inside $W(y)$ for some constants $\tau'>0$ and $\theta'\in (0,1)$.
Applying the Koebe principle again to the diffeomorphism $f^{S_{i(j)}-1}:\hJ_{i(j)}\to T^m$ and using the non-flatness of critical point, we obtain the desired estimate.

{\em Case 2.} $y\in V$. In this case, applying Lemma~\ref{lem:fromV} to $y$ and $n=i-i(j)$, we obtain that
$|f^{S_{i(j)}}(J_i)|\le \theta_0^{i-i_j}|T^m|$.
So $f^{S_{i(j)}}(J_i)$ is $C\tilde{\theta}_0^{i_j-i}$-well inside $T^m$ for some constants $C>0$ and $\tilde{\theta}_0\in (0,1)$.
Applying the Koebe principle again to the diffeomorphism $f^{S_{i(j)}-1}:\hJ_{i(j)}\to T^m$ and the non-flatness of critical point, we obtain the desired estimate.
\end{proof}

\section{Pull back of empty space}\label{sec:emptyspace}
In this section, we will assume that $f$ is non-renormalizable and that $\omega(c)$ is a wild attractor.  We fix a suitable neighborhood $\Lambda$ of $\omega(c)$. For each small nice interval $T\ni c$, let $\Lambda(T)$ denote the set of points in $T$ which return to $T$ before escaping $\Lambda$. Then we define a parameter $\xi(T)$ which measures the relative size of the complement of $\Lambda(T)$ (``empty space'') in $T$, and study the distortion of this parameter under pull back by $f$. The main results are Propositions~\ref{prop:emptyspace} and~\ref{prop:central}.

Fix a small symmetric nice interval $I\ni c$ such that the union $\Lambda$ of components of $D(I)\cup I$ which intersects $\omega(c)$ satisfies $\Lambda\Subset [f^2(c), f(c)]$.
For $T\subset I$, let
$$\Lambda(T)=\{x\in [0,1]: \exists k\ge 1 \text{ such that } x, f(x), \ldots, f^{k-1}(x)\in\Lambda \text{ and }f^k(x)\in T\}.$$

For an interval $Y\subset [0,1]$ and a set $Y'\subset Y$, we will define a number $\l(Y'|Y)$ to measure how much the subset $Y'$ occupies in $Y$. Let $\mathcal F_Y$ be the set of diffeomorphisms of the form $f^s:J\to Y$.
Define
$$\l(Y'|Y)=\sup_{\phi\in\mathcal F_Y}\frac{|\phi^{-1}(Y')|}{|\phi^{-1}(Y)|},$$
and $$\xi(Y'|Y)=1-\l(Y'|Y).$$
For a nice interval $T\subset I$ with $T\ni c$, define
$$\xi(T)=\xi(\Lambda(T)\cap T|T).$$

{\em Remark.} For each small nice interval $T$, $\xi (T)>0$. Indeed, $f$ is topologically transitive on $[f(c), f^2(c)]\supset T$, so $T\setminus \Lambda(T)$ has non-empty interior. Moreover, since $\omega(c)$ is minimal, $\partial T\cap \omega(c)=\emptyset$, so there exists $\delta=\delta(T)>0$ such that each diffeomorphism $f^s: J\to T$ extends to a diffemorphism onto the $\delta$-neighborhood of $T$. By the Koebe principle, there exists a constant $C=C(T)>0$ such that
$$\frac{|(f^{s}|J)^{-1} (T \setminus \Lambda(T))|}{|J|}\ge C \frac{|T\setminus \Lambda(T)|}{|T|}>0.$$

\begin{lem}\label{lem:zeroemptyspace}
Suppose that $f$ has a wild attractor. Let $T_n\ni c$ be a sequence of nice intervals such that $|T_n|\to 0$ as $n\to \infty$. Then $\xi(T_n)\to 0$.
\end{lem}

\begin{proof}
Assume by contradiction that there exists a sequence of nice intervals $T_n\ni c$ and a constant $\lambda>0$ such that $|T_n|\to 0$ and $\xi(T_n)\ge \lambda$, $n=1,2,\ldots$.
Since $f$ has a wild attractor, the non-escaping set $$\mathcal{J}=\{x\in [0,1]: f^k(x)\in \Lambda\text{ for all } k=0,1,\ldots\}$$ has positive Lebesgue measure.
Let $X=\{x\in \mathcal{J}:\omega(x)\ni c\}$. Then by M\~an\'e's theorem \cite{mane}, $|X|=|\mathcal{J}|>0$.
Let $n_0$ be large such that $|X\setminus T_{n_0}|>0$ and  let $x\in X \setminus T_{n_0}$ be a Lebesgue density point of $X$. For each $n\ge n_0$, let $s_n$ be the first entry time of $x$ under $f$ to $T_n$ and let $J_n=\mathcal{L}_x(T_n)$.
Then $f^{s_n}: J_n \to T_n$ is a diffeomorphism.
Since $f^{s_n}(X)\subset \Lambda(T_n)$,
we have
$$\frac{|J_n\cap X|}{|J_n|}\le \frac{|J_n\cap f^{-s_n}(\Lambda(T_n))|}{|J_n|}\le \lambda(T_n\cap \Lambda(T_n)|T_n) \le 1-\lambda.$$
Since $f$ has no wandering interval \cite{MS}, $|J_n|\to 0$. This contradicts  the assumption that $x$ is a Lebesgue density point of $X$.
\end{proof}

The following lemma is an improvement of \cite[Lemma 4.11]{li-shen-Haudoff}.

\begin{lem}\label{lem:capacity}
Let $T$ be small interval. Let $U_1, U_2,\ldots$ and $W_1,W_2, \cdots, W_k$ be pairwise disjoint
subintervals of $T$ and $Y\subset (\cup_i U_i)\cup (\cup_{i=1}^k W_i)$. Assume that
\begin{itemize}
\item for each $i$, $\lambda(Y\cap U_i|U_i)\le \lambda$;
\item for each $1\le j\le k$, $W_j$ is $\tau$-well inside $T$.
\end{itemize}
Then there exists $\eps=\eps(k,\tau)\in (0,1)$ such that
\begin{equation}\label{eqn:capacity}
1-\lambda(Y|T)\ge (1-\eps)(1-\lambda).
\end{equation}
Moreover, for a fixed $k$, $\eps(k,\tau)=O(\tau^{-1})$ as $\tau\to\infty$.
\end{lem}

\begin{proof}
Let $\mu_k(\tau)= (1+\tau^{-1})^{-k}$.
We first prove that
\begin{equation}\label{eqn:estW}
\sum_{j=1}^k |W_j|\le (1-\mu_k(\tau)) |T|.
\end{equation}
Without loss of generality, we may assume that $W_1, W_2, \ldots, W_k$ lie from left to right in $T$.
Let $\delta_j=|W_j|/|T|$ and let $\rho=1-\sum_{j=1}^k \delta_j$. Since the left component of $T\setminus W_j$ has length at least $\tau |W_j|$, we obtain
$$\rho \ge \tau \delta_1,$$
and for each $j=2,3,\ldots, k$,
$$\rho + \sum_{j'=1}^{j-1}\delta_{j'}  \ge \tau \delta_j.$$
By induction, it follows that for each $j=1,2,\ldots, k$,
$$\delta_j\le \frac{\rho}{\tau}\left(1+\tau^{-1}\right)^{j-1}.$$
Since $\rho+\delta_1+\cdots+\delta_k=1$, this implies that
$$\rho\ge (1+\tau^{-1})^{-k}=\mu_k(\tau),$$
hence $$\delta_1+\delta_2+\cdots+\delta_k\le 1-\mu_k(\tau).$$
The inequality (\ref{eqn:estW}) is proved.

Now let  $\tau'=0.9\frac{\tau^2}{1+2\tau}$ and $\eps(k,\tau)=1-\mu_k(\tau')$. Clearly, for a fixed $k$, $\eps(k,\tau)=O(\tau^{-1})$ as $\tau\to\infty$. It remains to show that (\ref{eqn:capacity}) holds with $\eps=\eps(k,\tau)$.
To this end, take an arbitrary diffeomorphism $\phi:T'\to T$ from the class $\mathcal F_T$. Let $U_i', W_j',Y'$ be the pre-images of $U_i, W_j ,Y$ under $\phi$ respectively.
By the Koebe principle, $W_j'$ is $\tau'$-well inside $T'$. Therefore as above, we obtain
$$\sum_{i=1}^k |W_j'| \le (1-\mu_k(\tau'))|T'|. $$
For each $i\ge 1$, $$|Y'\cap U_i'|\le \lambda(Y\cap U_i|U_i) |U_i'|\le \lambda |U_i'|.$$
Putting $U'=\bigcup U_i'$ and $W'=\bigcup_j W_j'$, we have $Y'\subset U'\cup W'$. Thus
\begin{align*}
|T'\setminus Y'| & \ge \sum_i |U_i'\setminus Y'|+|T'\setminus\left(U'\cup W'\right)| \\
                 & \ge (1-\lambda)(|U'|+ |T'\setminus(U'\cup W')|) \\
                 &=(1-\lambda)|T'\setminus W'| \\
                 &\ge (1-\lambda)\mu_k(\tau')|T'|.
\end{align*}
The inequality (\ref{eqn:capacity}) follows. 
\end{proof}

\begin{lem}\label{lem:capacity'}
Let $T$ be a small interval, let $T'$ be a unimodal pull back of $T$ by $f^s$, and let $Y'\subset T'$.
Assume that $Y:=f^s(Y')$ is covered by subintervals $U_i$ of $T$, $i=0,1,2,\ldots,$ such that
\begin{itemize}
\item for each $i\ge 0$, $U_i$ is $\tau$-well inside $T$;
\item for each $i\ge 1$, $\lambda(Y|U_i)\le \lambda$.
\end{itemize}
Then
\begin{equation}\label{eqn:capacity'}
1-\lambda(Y'|T')\ge (1-\eps(\tau)) (1-\lambda).
\end{equation}
Moreover, $\eps(\tau)=O(\tau^{-1/\ell})$ as $\tau\to\infty$.
\end{lem}
\begin{proof}
For each $i\ge 1$, $f^{-s}(U_i)\cap T'$ has at most two components and if $f^s(c)\in U_i$ then
$f^{-s}(U_i)\cap T'$ is an interval. Let $U_j'$ be the components of $f^{-s}(\bigcup_{i=0}^\infty U_i)$
such that for each $j\ge 4$, $f^s| U_j'$ is a diffeomorphism onto $U_i$ for some $i\ge 1$, hence
$$\lambda(Y'|U_j')\le \lambda(Y|U_i)\le \lambda.$$
By the Koebe principle, each $U_j'$ is $\tau'$-well inside $T'$, where $\tau'$ is a constant depending only on $\tau$ and $\tau'=O(\tau^{1/\ell})$ as $\tau\to\infty$.
 Thus by Lemma~\ref{lem:capacity}, the statement follows.
\end{proof}

\begin{lem}\label{lem:emptyspace}
Let $T\subset I$ be a nice interval that contains $c$ such that $T^1$ is $\tau$-well inside $T$ and  let $K\subset T\setminus T^1$ be a component of $\Lambda(T)$. Then
there exists $\eps=\eps(\tau)>0$, such that
\begin{equation}\label{equa:emptyspace}
\l(K\cap \Lambda(T^1)| K) \le \frac{\eps}{\xi(T)+\eps}.
\end{equation}
Moreover, $\eps(\tau)=O(\tau^{-1})$ as $\tau\to\infty$.
\end{lem}

\begin{proof}
Let $U_0=T\setminus \Lambda(T)$, $V_0=(\Lambda(T)\cap T)\setminus T^1$ and
$W_0=T^1$. Moreover, for each $k\ge 1$, inductively define
\begin{align*}
U_k& =\{x\in V_{k-1}: R_T^k(x)\in U_0\},\\
V_k& =\{x\in V_{k-1}: R_T^k(x)\in V_0\},\\
W_k& =\{x\in V_{k-1}: R_T^k(x)\in W_0\}.
\end{align*}

Since $T^1$ is $\tau$-well inside $T$, by the Koebe principle, there exists $\eps=\eps(\tau)>0$ such that
for each $\psi\in \mathcal{F}_T$, we have
$$|\psi^{-1}(T^1)|\le \eps |\psi^{-1}(T)|,$$
where $\eps=\eps(\tau)=(1+2\theta)^{-1}$ and $\theta=0.9\tau^2/(1+2\tau)$. So
$\eps(\tau)=O(\tau^{-1})$ as $\tau\to\infty$.

For each component $J$ of $V_{k-1}$, $R_T^k|J$ is a diffeomorphism
onto $T$, so for each $\phi\in\mathcal{F}_J$, we have $R_T^k|J\circ \phi\in \mathcal{F}_T$. Therefore,
$$|\phi^{-1}(U_k\cap J)|\ge \xi(T)|\phi^{-1}(J)|,$$
and $$|\phi^{-1}(W_k\cap J)|\le \eps |\phi^{-1}(J)|.$$ So
$$\frac{|\phi^{-1}(U_k\cap J)|}{|\phi^{-1}(W_k\cap J)|}\ge \frac{\xi(T)}{\eps}.$$
By Ma\~n\'e's Theorem \cite{mane}, $\bigcap_k V_k$ has measure
zero. It follows that
$$\frac{|\phi^{-1}(T\cap W)|}{|\phi^{-1}(T)|}\le \frac{\eps}{\xi(T)+\eps},$$
where $$W:=\bigcup_{k=0}^\infty W_k.$$
Thus,
$$\lambda (T\cap W | T)\le \frac{\eps}{\xi(T)+\eps}.$$
For each component $K$ of $V_0$, since the first return map $R_T$ maps $K\cap \Lambda(T^1)$ onto $W$, we have
$\lambda(K\cap W|K)\le\lambda (T\cap W|T)$. The lemma follows.
\end{proof}

\begin{prop}\label{prop:emptyspace}
For any $\tau>0$, there exists $\eps=\eps(\tau)\in (0,1)$, such that for any $\tau$-nice interval $T\subset I$ with $T\ni c$ and any child $J$ of $T$,
$$\frac{\xi(J)}{\xi(T)}\ge \frac{1-\eps}{\xi(T)+\eps}.$$
Moreover,  $\eps(\tau)=O(\tau^{-1/\ell})$ as $\tau\to\infty$.
\end{prop}

\begin{proof}
Let $s$ be a transition time of $J$ to $T$. By Lemma~\ref{lem:childtime} $Y:=f^s(\Lambda(J)\cap J)\subset \Lambda(T^1)\cup T^1.$
Let $U_0, U_1, \ldots$ be the components of $\Lambda(T)\cap T$ such that $U_0\ni c$. Then for all $i\ge 0$, $U_i$ is $\tau$-well inside $T$.
By Lemma~\ref{lem:emptyspace}, for each $i\ge 1$,
$\xi (Y|U_i)\ge \xi(T)/(\xi(T)+\eps_1)$, where $\eps_1=\eps_1(\tau)=O(\tau^{-1})$ as $\tau\to\infty$.
By Lemma~\ref{lem:capacity'}, the statement follows.
\end{proof}

The previous proposition says that the empty space of a unimodal pull back does not decrease too much. Now we will show that the central cascade does not influence the empty space too much as well.

\begin{definition}
Given a maximal central cascade $T\supset T^1\supset\cdots \supset T^m$, an {\em inheritor} of $T$ is, by definition, a child $J$ of $T^{m'}$ for some $0\le m'\le m$ such that $J\subsetneq T^m$.
\end{definition}

\begin{prop}\label{prop:central}
Let $T\supset T^1\supset \dots \supset T^m$ be a maximal central cascade, where $T\ni c$ is a small symmetric $\tau$-nice interval.
Then there is a constant $C=C(\tau)>0$ such that for each inheritor $J$ of $T$, we have
$$\xi(J)\ge C \xi(T).$$
\end{prop}

\newcommand{\hK}{\widehat{K}}
\begin{proof}
Let us first prove the proposition under the following assumption:
$$(*) \mbox{ each component of $D(T)\cap (T\setminus T^1)$ is $\tau$-well inside $T\setminus T^1$.}$$
Let $E_T$, $F_T$, $V,$ $Q,$ and $X$ be as defined in \S~\ref{sec:sizeofchildren}. Let $V_0$ be the component of $T^m\setminus \overline{Q}$ which contains $c$. Let $Q'=Q\cap D(X\cup V_0)$. Note that $X\cup V_0$ is a nice set and for each component $K$ of $Q'$, the first entry time of $K$ into $X\cup V_0$ is of the form $ns$ and $f^{ns}$ maps a neighborhood of $K$ in $Q$ diffeomorphically onto $T^m$.

{\bf Claim 1.} There exists a constant $\tau'>0$ such that
\begin{enumerate}
\item[(1a)] for each $0\le i<m$, each component of $E(T)\cap (T^i\setminus T^{i+1})$ is $\tau'$-well inside $T^i$;
\item[(1b)] $V$ is $\tau'$-well inside $V_0$;
\item[(1c)] each component of $X$ is $\tau'$-well inside $T^m$;
\item[(1d)] each component of $Q'$ is $\tau'$-well inside $T^m$.
\end{enumerate}
\begin{proof}[Proof of Claim 1.] (1a). By assumption, if $K$ is a component of $E(T)\cap (T\setminus T^1)$, then
$K$ is $\tau$-well inside $T\setminus T^1\subset T$.
For each $1\le i<m$,
$f^{is}$ maps each component of $T^i\setminus T^{i+1}$ diffeomorphically onto $T\setminus T^1$ and for each component $K$ of
$E(T)\cap (T^i\setminus T^{i+1})$, $K'=f^{is}(K)$
is a component of $E(T)\cap (T\setminus T^1)$. The statement follows by the Koebe principle.

(1b). Note that $V_0$ is a unimodal pull back of the component of $T^{m-1}\setminus T^m$ which contains $f^s(c)$. Since $f^s(V)\subset E(T)$, by (1a), $f^s(V)$ is well inside $T^{m-1}\setminus T^m$.
Thus the statement follows by the Koebe principle and non-flatness of critical point. (We need to redefine the constant $\tau'$.)

(1c). It also follows from (1a) by the Koebe principle and non-flatness of critical point.

(1d) follows from (1b) and (1c) and the observation on the components $Q'$ by the Koebe principle.
\end{proof}
{\bf Claim 2.}
There exists a constant $C_0>0$ such that
\begin{enumerate}
\item[(2a)]  for each component $K$ of $E(T)\cap (T\setminus T^1)$,
we have
\begin{equation*} 
\xi(\Lambda(T^1)\cap K|K)\ge C_0\xi(T); 
\end{equation*}
\item [(2b)] for each $1\le i<m$ and each component $K$ of $E(T)\cap (T^i\setminus T^{i+1})$, we have
\begin{equation*} 
\xi(\Lambda(T^i)\cap K|K)\ge C_0\xi(T); 
\end{equation*}
\item[(2c)] for each component $K$ of $X$, we have
\begin{equation*} 
\xi(\Lambda(V)\cap K|K)\ge C_0\xi(T);
\end{equation*}
\item[(2d)] for the interval $V_0$, we have
\begin{equation*} 
\xi((\Lambda(V)\cup V)\cap V_0|V_0)\ge C_0\xi(T);
\end{equation*}
\item[(2e)] for each component $K$ of $Q'$, we have
 \begin{equation*} 
\xi(\Lambda(V)\cap K|K)\ge C_0\xi(T).
\end{equation*}
\end{enumerate}
\begin{proof}[Proof of Claim 2.] (2a) follows from Lemma~\ref{lem:emptyspace}.

(2b) follows from (2a) and the observation that $K'=f^{is}(K)$ is a component of $E(T)\cap (T\setminus T^1)$ and
$f^{is}(K\cap \Lambda(T^i))\subset K'\cap \Lambda(T^i)\subset K'\cap \Lambda(T^1)$.

(2c) follows similarly.

(2d). The set $(\Lambda(V)\cup V)\cap V_0$ is covered by $V$ and the components of $X\cap V_0$. The statement follows from (1b) and (2c) by Lemma~\ref{lem:capacity}.

(2e) follows from (2c) and (2d) by the observation on $Q'$.
\end{proof}

Now suppose $J\subsetneq T^m$ is a child of $T^{m'}$ for some $0\le m'\le m$.
Let $t$ be a transition time from $J$ to $T^{m'}$. Note that $J\subset V$, so
$$Y:=f^t(\Lambda(J)\cap J)\subset (\Lambda(J)\cup J)\cap T^{m'}\subset (\Lambda(V)\cup V)\cap T^{m'}.$$
Let $U_0, U_1, U_2, \ldots$ be the components of $E(T)\cap T^{m'}$, $X\setminus \overline{Q\cup V_0}$, $V_0$ and $Q'$.
These sets cover $Y$. Each of these intervals are uniformly well inside $T^{m'}$ and $\xi(Y|U_i)\ge C_0\xi(T)$. By Lemma~\ref{lem:capacity'}, it follows that
$\xi(J)\ge C\xi(T),$ where $C>0$ is a constant.

We have completed the proof of the proposition under the assumption (*). For the general case, by Proposition~\ref{prop:emptyspace}, we may assume $m\ge m'\ge 2$, so $T^1\supset T^2\supset \cdots T^m$ is also a maximal central
cascade. We claim that each component $K$ of $D(T^1)\cap (T^1\setminus T^2)$ is $\tau_1$-well inside $T^1\setminus T^2$ for some $\tau_1>0$.
Indeed, the first return time of $K$ into $T^1$ is greater than $s$. Since $T^1$ is well inside $T$, $f^s(K)$ is well inside a component $K'$ of $D(T)$, by Theorem~\ref{thm:SV}. Thus $f^s(K)$ is well inside $T\setminus T^1$, which implies that $K$ is well inside $T^1\setminus T^2$. Applying the above argument to the maximal central cascade $T^1\supset T^2\supset \cdots \supset T^m$ proves the statement.
\end{proof}

\section{Proof of the Reduced Main Theorem}\label{sec:reducedmainthm}
We continue to assume that $f$ is non-renormalizable and has a Cantor attractor $\omega(c)$. Fix a neighborhood $\Lambda$ of $\omega(c)$ as in the previous section.
Let $Y\ni c$ be a symmetric nice interval which is necessarily contained in $(f(c), f^2(c))$. Let $\mathcal{N}_Y=\{n\ge 1: f^n(c)\in Y\}$ and for each $n\in\mathcal{N}_Y$, let $Y_{-n}$ denote the pull back of $Y$ by $f^n$ which contains $c$. To obtain an upper bound for the number of children of $Y_{-n}$, we first apply Propositions~\ref{prop:sizeofchildren}, ~\ref{prop:emptyspace} and ~\ref{prop:central} to obtain lower bounds on $\xi(Y_{-n})$, together with niceness control on young children. Then we apply Proposition~\ref{prop:emptyspace} again to obtain the desired upper bound: if $Y_{-n}$ has too many children, then some grandchild of $Y_{-n}$ has a large ``empty space'' which is ruled out by Lemma~\ref{lem:zeroemptyspace}.

Since $f$ is non-renormalizable, we have
\begin{equation}\label{eqn:t-n2zero}
\lim_{\substack{n\in\mathcal{N}_Y\\ n\to\infty}} |Y_{-n}|= 0.
\end{equation}
For $n\in\mathcal{N}_Y$ with $n\ge 1$, we shall define a positive integer $M_n(Y)$, called the {\it essential order} of $Y_{-n}$. 
Let $\{Y_i\}_{i=-n}^0$ be the chain with $Y_0=Y$. Let $0=i_0>i_1>\dots>i_p=-n$ be all the integers such that $Y_{i_j}\ni c$. So $Y_{i_j}$ is a child of $Y_{i_{j-1}}$ with the transition time $s_j=i_{j-1}-i_{j}$,  for each $1\le j \le p$. By Lemma~\ref{lem:childtime}, $s_1 \le s_2 \le \dots \le s_p$. Define
$$M_n(Y)=\#\{s_j: 1\le j\le p\}.$$
Let $$\mathcal{N}_M(Y)=\{n\in\mathcal{N}_Y: M_n(Y)\le M\}.$$

\begin{prop}\label{prop:order}
There exists a universal constant $C_0>0$ such that for any symmetric nice interval $Y\ni c$ and $n\in\mathcal{N}_Y$, we have $M_n(Y)\le C_0\log n$.
Moreover, there exists a universal constant $\kappa$ such that the transition time from the second child of $Y_{-n}$ to $Y_{-n}$ is greater than $\kappa n$.
\end{prop}
\begin{proof}
Let $n'>n$ be such that $n'\in \mathcal{N}_Y$ and such that $Y_{-n'}$ is the second child of $Y_{-n}$. Let $\{Y_i\}_{i=-n'}^0$, $\{i_j\}_{j=0}^{p'}$ and $\{s_j\}_{j=1}^{p'}$ be defined as above, and let $M'=M_{n'}(Y)$. Note that $M:=M_n(Y)=M'-1$. Define $m(0)=0$,  $m(1)=1$, and define inductively integers $m(1)<m(2)<\dots<m(M')\le p'$ by
$$m(j)=\inf \{m>m(j-1): s_m>s_{m(j-1)}\}, \, \, j=2,3,\ldots.$$
For $1\le j\le M$, let $r_j$ denote the minimal return time of points in $Y_{i_{m(j)-1}}\cap \omega(c)$ to $Y_{i_{m(j)-1}}$.
Let us show that for each $2\le j\le M'$,
\begin{equation}\label{eqn:smrec}
s_{m(j)}\ge s_{m(j-1)}(m(j)-m(j-1))+ r_{j-1}.
\end{equation}
Indeed, in the case $m(j)=m(j-1)+1$, this is clear as $s_{m(j)}-s_{m(j-1)}$ is a return time of $f^{s_{m(j-1)}}(c)$ to $Y_{i_{m(j-1)-1}}$ 
When $m(j)-m(j-1)>1$, observe that $f^{s_{m(j-1)}}(c)\in Y_{i_{m(j)-2}}\setminus Y_{i_{m(j)-1}}$ and hence $$f^{k s_{m(j-1)}}(c)\in Y_{i_{m(j)-k-1}}\setminus Y_{i_{m(j)-k}}$$ for each $1\le k\le m(j)-m(j-1)$.
Since $Y_{i_{m(j)}}$ is a child of $Y_{i_{m(j)-1}}$, we have $s_{m(j)}> (m(j)-m(j-1)) s_{m(j-1)}$ and that $s_{m(j)}-(m(j)-m(j-1)) s_{m(j-1)}$ is a return time of $f^{(m(j)-m(j-1)) s_{m(j-1)}}(c)$ to 
$Y_{i_{m(j-1)-1}}$. The inequality (\ref{eqn:smrec}) follows.

By Lemma~\ref{lem:childtime} , for $2\le j\le M'$, $r_j\ge s_{m(j-1)}$. Thus for $3\le j\le M'$,
\begin{equation}\label{eqn:recsm}
s_{m(j)}\ge s_{m(j-1)}(m(j)-m(j-1))+s_{m(j-2)}> s_{m(j-1)}+s_{m(j-2)}.
\end{equation}
Thus $s_{m(j)}$ grows at least as fast as the Fibonacci sequence. Since $n\ge s_{m(M)}$,
it follows that $M\le C_0\log n$ for some universal constant $C_0>0$.

To prove the last statement, note that $s_{m(M')}$ is the transition time from $Y_{-n'}$ to $Y_{-n}$. Write $S_j=s_{m(j)}(m(j+1)-m(j))$. Then (\ref{eqn:recsm}) implies that $S_j$ is strictly increasing in $j$ and $S_{j+1}\ge S_j+S_{j-2}$. As $n=S_1+S_2+\cdots+S_M$, it follows that $S_M/n$ is bounded away from zero. By (\ref{eqn:recsm}) again, $n'-n=s_{m(M+1)}$ is at least comparable to $n$.
\end{proof}
\begin{remark} In the proof, we only used that $f$ is non-renormalizable. Thus this proposition holds whenever $f$ is non-renormalizable.
\end{remark}
This proposition allows us to obtain a lower bound on $q(\mathcal{Y},n)$ for a nice cover $\mathcal{Y}$, which implies a lower bound for the topological complexity function for maps with special combinatorics. See Theorem~\ref{thm:lower} and Corollary~\ref{cor:special} at the end of this section.

\begin{lem}\label{lem:finite} Given a symmetric nice interval $Y$, for each $M\ge 1$, we have $\#\mathcal{N}_M(Y)<\infty$.
\end{lem}
\begin{proof}
It suffices to prove that $\mathcal{N}_1(Y)$ is finite, since for each $n\in\mathcal{N}_M(Y)$ with $M\ge 2$, there exists $n'\in\mathcal{N}_{M-1}(Y)$ such that $n\in\mathcal{N}_1(Y_{-n'})$.

By Proposition~\ref{prop:persistentrec}, the set
$$\mathcal{N}_1^o(Y)=\{s: Y_{-s} \mbox{ is a child of } Y\}$$
is finite. For each $n\in \mathcal{N}_1(Y)\setminus \mathcal{N}_1^o(Y)$, there exists $s\in\mathcal{N}_1^o(Y)$ such that
$n-s\in \mathcal{N}_1(Y)$ and such that $Y_{-n}$ is a child of $Y_{-n+s}$. Since
$|Y_{-n+s}|\ge |f^{s}(c)-c|$ is bounded away from zero, by (\ref{eqn:t-n2zero}), $n-s$ is bounded from above.  It follows that $\mathcal{N}_1(Y)\setminus \mathcal{N}_1^o(Y)$, hence $\mathcal{N}_1(Y)$, is finite.
\end{proof}

In particular, for each symmetric nice interval $Y\subset I$,
$$\widehat{\xi}_M(Y)=\inf \{\xi(Y_{-n}): n\in \mathcal{N}_M(Y)\}>0.$$
\begin{lem} \label{lem:criticalpb}
Let $Y\ni c$ be a small symmetric nice interval, let $n\in\mathcal{N}_Y$ be such that $M_n(Y)\ge 3$,
and let $Y_{-n}$ be the critical pull back of $Y$ by $f^n$.
Let $K_1\supset K_2\supset \cdots$ be all the children of $Y_{-n}$. Then
for each $k\ge 2$, $K_k$ is $C\lambda_0^{-k}$-well inside $Y_{-n}$ and
\begin{equation}\label{xiKk}
\xi(K_k)\ge n^{-\beta} \widehat{\xi}_2(Y),
\end{equation}
where $C>0$, $\lambda_0\in (0,1)$ and $\beta>0$ are universal constants.
\end{lem}
\begin{proof}
Let $\{Y_i\}_{i=-n}^0$ be the chain with $Y_0=Y$ and define $m(0), m(1), \ldots$ as above. Define
$T_j=Y_{i_j}$ for $0\le j \le p$. Note that $T_{m(2)}$ is of the form $Y_{-n'}$ for some $n'\in \mathcal{N}_2(Y)$, so
$$\xi(T_{m(2)})\ge \widehat{\xi}_2(Y).$$

Let us first prove that there exists a universal constant $\tau>0$ such that for each $2\le j\le M$, $T_{m(j)}$ is a $\tau$-nice interval.
If either $T_{m(j)}$ is well inside $T_{m(j)-1}$ or
$T_{m(j)-1}$ is well inside $T_{m(j)-2}$ then by Lemma~\ref{lem:wellinsidenice} we are done. Thus, by Lemma~\ref{lem: boundsofchild}, we may assume that $T_{m(j)}$ is the first child of $T_{m(j)-1}$ and that $T_{m(j)-1}$ is the first child of $T_{m(j)-2}$, i.e.
$s_{m(j)}$ is the first return time of $c$ into $T_{m(j)-1}$ and $s_{m(j)-1}$ is the first return time of $c$ into $T_{m(j)-2}$. Since $s_{m(j)}> s_{m(j)-1}$, it follows that $R_{T_{m(j)-2}}: T_{m(j)-1} \to T_{m(j)-2}$ is non-central. By Theorem~\ref{thm:realbounds}, it follows that $T_{m(j)}$ is uniformly well inside $T_{m(j)-1}$ and thus we are done.

Now let us show that there exists $\kappa\in (0,1)$
such that $\xi(T_{m(j)})\ge \kappa\xi(T_{m(j-1)})$ for each $3\le j\le M$. Indeed, by Proposition~\ref{prop:emptyspace}, such an estimate holds if $m(j)= m(j-1)+1$. So assume $m(j)> m(j-1)+1$. By Lemma~\ref{lem:childtime}, it follows that $T_{m(j-1)}\supset T_{m(j-1)+1}\supset \cdots \supset T_{m(j)-1}$ is a central cascade, i.e., $s_k=s_{m(j-1)}$ is the first return time to $T_{k-1}$ for each $m(j-1)<k\le m(j)-1$.
Since $s_{m(j)}> s_{m(j-1)}$, $T_{m(j)}$ is an inheritor of $T_{m(j-1)}$.
So by Proposition~\ref{prop:central}, the statement follows.

Similarly for each $k\ge 2$, $K_k$ is either a child or an inheritor of $T_{m(M)}$, so $\xi(K_k)\ge \kappa \xi(T_{m(M)})$. Thus
$$\xi(K_k)\ge \kappa^{M-1} \xi(T_{m(2)})\ge \kappa^{M-1} \widehat{\xi}_2(Y).$$
By Proposition~\ref{prop:order}, the statement follows.

If $Y_{-n}=T_{m(M)}$, then by Lemma~\ref{lem:largebounds}, the children of $Y_{-n}$ are well nested, so the niceness of $K_k$ follows from Lemma~\ref{lem:wellinsidenice}.  If $Y_{-n}\subsetneq T_{m(M)}$, then for each $k\ge 2$, the same conclusion follows from Proposition~\ref{prop:sizeofchildren}. 
\end{proof}

\begin{proof}[Proof of the Reduced Main Theorem]
By (\ref{eqn:t-n2zero}) and Lemma~\ref{lem:finite}, we may assume that $Y$ is small so that Lemma~\ref{lem:criticalpb} applies.

Let $n\in \mathcal N_Y$ be so large that $M_n(Y)\ge 3$ and $Y_0$ be the critical pull back of $Y$ under $f^n$. Assume the number $N$ of children of $Y_0$ is at least $2$ and let $K_N$ denote the $N$-th child of $Y_0$. Let $L$ be the first child of $K_N$.
Then by Lemma~\ref{lem:criticalpb}, $\xi(K_N)\ge n^{-\beta} \widehat{\xi}_2(Y)$ and
$K_N$ is a $C\lambda_0^{-N}$-nice interval, where $C>0$ and $\lambda_0\in (0,1)$ are universal constants.
So by Proposition~\ref{prop:emptyspace}, we have
$$\xi(L)\ge \frac{1-C_0\lambda_0^{N/\ell}}{\xi(K_N)+C_0\lambda_0^{N/\ell}}\xi(K_N)\ge \frac{1-C_0\lambda_0^{N/\ell}}{\widehat{\xi}_2(Y)+C_0n^\beta \lambda_0^{N/\ell}} \widehat{\xi}_2(Y),$$
where $C_0>0$ is a constant. 
On the other hand, by Lemma~\ref{lem:zeroemptyspace}, when $n$ is large enough, we have $\xi(L)\le \widehat{\xi}_2(Y)/2$. Since $\widehat{\xi}_2(Y)\le 1$, it follows that $N=O (\log n)$.
\end{proof}

We end this section with the following theorem.
\begin{thm}\label{thm:lower} Let $f\in\mathcal{A}_*$ be a non-renormalizable unimodal map with a non-periodic recurrent critical point $c$ and such that $\omega(c)$ is minimal. Let $Y$ be a symmetric nice interval, let $\mathcal{Y}$ denote the collection of components of $\dom(Y)\cup Y$ which intersects $\omega(c)$ and let $q(\mathcal{Y},n)$ be defined as in (\ref{eqn:qclY}).
Then 
$$q(\mathcal{Y}, n)\ge \kappa_0 n$$
holds for all $n$ large enough,
where $\kappa_0>0$ is a universal constant.
\end{thm}
\begin{proof} Let $N$ denote the maximal entry time of a point in $Y\cap\omega(c)$ into $Y$ and
for each $n\ge 0$, let $T_n$ denote the connected component of $f^{-n}(\dom(Y)\cup Y)$ which contains $c$.
Clearly, $T_0\supset T_1\supset \cdots$.

Fix a large positive integer $n$, and let $m_0$ be minimal positive integer such that $n-m_0<\kappa m_0,$ where $\kappa>0$ is as in Proposition~\ref{prop:order}. Then
$n-(m_0-1)\ge \kappa (m_0-1)$, and hence
$$n-m_0\ge \kappa m_0 -(1+\kappa)\ge \frac{\kappa n}{1+\kappa}-1-\kappa.$$
Let $m$ be the minimal integer such that $m\ge m_0$ and $f^{m}(c)\in Y$. Then $m<m_0+N$, 
so
$$n-m> \kappa_0 n,$$
provided that $n$ is large enough, where $0<\kappa_0<\kappa/(1+\kappa)$. By Proposition~\ref{prop:order}, the second child of $T_{m}$ has transition time greater than $\kappa m>n-m$.

Note that for any symmetric nice intervals $I\supset I'$, if $J$ is a child of $I$ with transition time $r$, then $I'$ has a child $J'$ with transition time at least $r$ such that $J'\subset J$. Indeed, if $s\ge 0$ is the minimal integer such that $f^{s+r}(c)\in I'$, then $I'$ has a child with transition time $s+r$.   Thus for each $m\le i< n$,
there exists $r_i$ such that
 \begin{itemize}
 \item $T_{m+r_m}$ is the second child of $T_m$;
 \item $T_{i+r_{i}}$ is a child of $T_i$, $m\le i<n$;
 \item $r_m\le r_{m+1}\le \cdots\le r_{n-1}$. 
 \end{itemize}
Therefore
$$n-1+r_{n-1}> n-2+r_{n-2}>\cdots> m+r_m>n.$$
For each $m\le i<n$, let $J_i$ denote the pull back of $T_i$ by $f^{n-i}$ which contains $f^{i+r_i-n}(c)$. Then each $J_i$ is a component of $f^{-n}(Y\cup \dom(Y))$ intersecting $\omega(c)$. Note that $f^{n-i}$ maps $J_i$ diffeomorphically onto $T_i\ni c$. It follows that
for $m\le i<i'<n$, $J_i\cap J_{i'}=\emptyset$, for otherwise, $J_i=J_{i'}$ is mapped onto $T_{i'}$ diffeomorphically by $f^{n-i'}$ and mapped onto $T_i$ diffeomorphically by $f^{n-i}$ which is absurd.   Therefore, $q_n(\mathcal{Y})\ge n-m\ge \kappa_0 n.$
\end{proof}

Let us say that a map $f\in\mathcal{A}_*$ has {\em special combinatorics} if $\omega(c)\ni c$ is minimal and there exists a symmetric nice interval $Y$ such that for each $n=0,1,\ldots$,
$Y_n\setminus Y_{n+1}$ has exactly one component intersecting $\omega(c)$, where $Y_0=Y$ and $Y_{n+1}$ is the return domain of $Y_{n}$ which contains $c$.
Such a map is necessarily non-renormalizable.

\begin{cor}\label{cor:special} Suppose that $f\in\mathcal{A}_*$ has special combinatorics. Then for any small open cover $\mathcal{U}$ of the Cantor set $\omega(c)$, the topological complexity function satisfies
$$p(\mathcal{U}, n+1) \ge \kappa_0 n \text{ for all } n\text{ large enough.}$$
\end{cor}
\begin{proof} Let $Y$ be a symmetric nice interval as in the definition above. We first show that if $J$ is a pull back of $Y$ which intersects $\omega(c)$ and $J\not\ni c$, then $\hJ\cap\omega(c)=\emptyset$, where $\hJ$ is the interval lying on different side of $c$ as $J$ and with $f(\hJ)=f(J)$.
Indeed, let $n\ge 0$ be maximal such that $J\subset Y_n$. So $J\not\subset Y_{n+1}$. As both $Y_{n+1}$ and $J$ are pull backs of $Y_0$ and $J\not\ni c$, we have $J\cap Y_{n+1}=\emptyset$. Since only one component of $Y_n\setminus Y_{n+1}$ intersects $\omega(c)$, it follows that $\hJ\cap\omega(c)=\emptyset$.

Let $\mathcal{Y}$ denote the nice cover associated with $Y$. Let us show that $q(\mathcal{Y},n)=p(\mathcal{Y}, n+1)$ for each $n\ge 0$. To this end, it suffices to show that each element of $\bigvee_{j=0}^n f^{-j}\mathcal{Y}$ has at most one component intersecting $\omega(c)$. Arguing by contradiction, assume that there exists $n\ge 0$ such that some element $K$ of $\bigvee_{j=0}^n f^{-j}\mathcal{Y}$  has at least two components intersecting $\omega(c)$. If $n$ is minimal with the last property, then $K$ has exactly two components intersecting $\omega(c)$ which are symmetric around $c$, which is impossible by what we proved in the previous paragraph.

By Theorem~\ref{thm:lower}, it follows that $p(\mathcal{Y},n+1)\ge \kappa_0 n$ for all $n$ large enough. For each open cover $\mathcal{U}$ which is a refinement of $\mathcal{Y}$, $p(\mathcal{U},n+1)\ge p(\mathcal{Y},n+1)\ge \kappa_0 n$.
\end{proof}
\begin{remark}\label{rem:Fibonacci} The well-studied Fibonacci unimodal maps have special combinatorics. See for example  \cite[Section 6]{LM}.
\end{remark}
\section{Appendix: A wild adding machine}\label{sec:appendix}
\begin{thm}\label{thm:ec} There exists a unimodal map $f\in \mathcal{A}_*$ which has a wild attractor $\omega(c)$ such that $f:\omega(c)\to\omega(c)$ is topologically conjugate to an adding machine and hence equicontinuous.
\end{thm}

Following~\cite{BS}, we define an adding machine as follows.
Let $\alpha = (p_1, p_2, \ldots)$ be a sequence of
integers where each $p_i \ge 2$. Let $\Delta_\alpha$ denote the set of all sequences $(x_1, x_2, \ldots)$, where
$x_i\in {0, 1, \ldots, p_i -1}$ for each $i.$ We use the product topology on $\Delta_\alpha$.
For each $\alpha$,  an adding machine map $f_\alpha: \Delta_\alpha\to\Delta_\alpha$ is defined as:
$$f_\alpha (x_1, x_2, \cdots)=\left\{
\begin{array}{ll}
(\overbrace{0,0,\ldots, 0}^{l-1}, x_l+1, x_{l+1},\cdots) & \mbox{ if } x_i=p_i-1\mbox{ for } i<l\\
& \,\,\,\,\mbox{ and } x_l<p_l-1\\
(0,0,0,\cdots) & \mbox{ if } x_i=p_i-1 \mbox{ for all }i.
\end{array}
\right.
$$
It is clear that $f_\alpha:\Delta_\alpha\to\Delta_\alpha$ is minimal and equicontinuous. It is well known that for an infinitely renormalizable map $f\in \mathcal{A}_*$, $f:\omega(c)\to\omega(c)$ is topologically conjugate to an adding machine, see \cite[Proposition III.4.5]{MS}. (The definition of an adding machine there is slightly different, but equivalent to the one above.)

In~\cite{BKM} the authors constructed uncountably many non-renormalizable unimodal maps such that $f:\omega(c)\to\omega(c)$ is topologically conjugate to a (generalized) adding machine, hence equi-continuous. It seems that their construction only gives non-persistent recurrent maps. To obtain a equi-continuous wild attractor, we shall modify their construction to obtain a unimodal map with a wild attractor.

We start with the following lemma which gives a sufficient condition for a non-renormalizable unimodal map for which $f:\omega(c)\to\omega(c)$ is topologically conjugate to an adding machine.
\begin{lem}\label{lem:eqvcon}
Let $f\in\mathcal{A}_*$ be a unimodal map with a recurrent critical point $c$.
Assume that for each $n=1,2,\ldots$, there exists a nice interval $T_n\ni c$ together with three distinct return domains $T_n', Q_n, \hQ_n$ such that 
\begin{enumerate}
\item[(i)] $|T_n|\to 0$;
\item[(ii)] $T_n'\ni c$ and $f(Q_n)=f(\hQ_n)$;
\item[(iii)] for each $x\in T_n'\cap\omega(c)$, $R_{T_n}(x)\in Q_n\cup \hQ_n$;
\item[(iv)] for each $x\in (Q_n\cup \hQ_n)\cap\omega(c)$, $R_{T_n}(x)\in T_n'$.
\end{enumerate}
Then $f:\omega(c)\to\omega(c)$ is topologically conjugate to an adding machine.
\end{lem}
\begin{proof} By passing to a subsequence if necessary, we may assume $T_n\supset T_{n+1}$ for each $n\ge 1$.
For each interval $n$, we shall construct a cover $\cU_n$ of $\omega(c)$, such that
\begin{itemize}
\item $\cU_n$ consists of consisting of finitely many pairwise disjoint close subsets of $\omega(c)$ which are cyclically permuted by $f$;
\item $\cU_{n+1}$ is a refinement of $\cU_n$;
\item the maximum diameter of elements of $\cU_n$ converges to $0$ as $n\to\infty$.
\end{itemize}
It is well-known that existence of such covers $\cU_n$ imply that $f:\omega(c)\to\omega(c)$ is topologically conjugate to some adding machine map. See for example \cite[Theorem 1.1]{BKM} and references therein.

To this end, let $k_n$ and $l_n$ denote the return time of $T_n'$ and $Q_n$ into $T_n$ respectively.
Define
$$U_j^n=\left\{
\begin{array}{ll}
f^{j}(T_n'\cap\omega(c)) & \mbox{ if } 0\le j<k_n;\\
f^{j-k_n}((Q\cup\hQ)\cap\omega(c)) & \mbox{ if } k_n\le j\le k_n+l_n,
\end{array}
\right.
$$
and let $\mathcal{U}_n=\{U_j^n:0\le j<k_n+l_n\}$.
Then the assumptions imply that $f(U_j^n)\subset U_{j+1}^n$ 
for each $0\le j<k_n+l_n$ and $U_{k_n+l_n}^n\subset U_0^n$. Moreover, condition (iv) implies that $\omega(c)\cap \partial T_n=\emptyset$, and hence $\omega(c)\cap\partial P=\emptyset$ for each entry domain $P$ of $T_n$. Thus each $U_j^n$ is a closed subset of $\omega(c)$, and $\bigcup_{j=0}^{k_n+l_n-1} U_j^n=\omega(c)$.
Let us show that $\{U_j^n\}_{j=0}^{k_n+l_n-1}$ are pairwise disjoint. If $0\le j_1<j_2<k_n$ or $k_n\le j_1<j_2< k_n+l_n$, then for any $a_1\in U_{j_1}$ and $a_2\in U_{j_2}$, $a_1$ and $a_2$ have different return times to $T_n$, thus $U_{j_1}^n\cap U_{j_2}^n=\emptyset$. If $0\le j_1<k_n\le j_2<k_n+l_n$, then for $a_1\in U_{j_1}$ and $a_2\in U_{j_2}$, $R_{T_n}(a_1)\in Q_n\cap \hQ_n$ is different from $R_{T_n}(a_2)\in T_n'$, thus we also have
$U_{j_1}^n\cap U_{j_2}^n=\emptyset$.

Since $T_{n+1}\subset T_n$, each component of $D(T_{n+1})\cup T_{n+1}$ is contained in a component of $D(T_n)\cup T_n$. It follows that $\cU_{n+1}$ is a refinement of $\cU_n$.  Since $f$ has no wandering interval, the supremum of length of components of $D(T_n)\cup T_n$ tends to $0$, thus the maximal diameter of elements of $\mathcal{U}_{T_n}$ tends to $0$.
\end{proof}

We shall now describe the combinatorial property of our example in terms of the first return map to symmetric nice intervals.
For a unimodal map $f$ with $f(0)=f(1)=0$ and $f(c)>c$, there is an orientation-reversing fixed point $q$. For $x\in [0,1]\setminus\{c\}$, let $\hat{x}$ denote the preimage of $f(x)$ other than $x$, and let $\hat{c}=c$. Let $I_0=(\hat{q}, q)$, and whenever $I_k$ is defined and $c$ returns to $I_k$, define $I_{k+1}$ to be the return domain of $I_k$ containing $c$. (So $I_0\supset I_1\supset \cdots$ be the principal nest starting from $I_0$.) Let $R_{I_k}$ denote the first return map to $I_k$. If $R_{I_k}(c)$ returns to $I_k$, let $J_{k+1}$ denote the return domain of $I_k$ containing $R_{I_k}(c)$, and let $\hJ_{k+1}=\{\hat{x}: x\in J_{k+1}\}$.
These objects depend of course on $f$, and when we want to emphasize the map $f$, we write $I_k^f$, $J_{k}^f$, $R_{I_k}^f$, etc.
\begin{prop} \label{prop:map}
There is a unimodal map $f:[0,1]\to [0,1]$ with the following properties:
\begin{enumerate}
\item $f(c)>c$, $f^2(c)<\hat{q}$ and $f^3(c),f^5(c)\in I_0$;
\item $I_{k+1}$ and $J_{k+1}$ are defined and disjoint for all $k=0,1,\ldots$;
\item $R_{I_k}(I_{k+1})\supset I_{k+1}\ni c$;
\item $R_{I_{k}}|I_{k+1}=R_{I_{k-1}}^2|I_{k+1}$ for each $k\ge 1$;
\item When $k$ is odd, $R_{I_{k-1}}(J_{k+1})\subset \hJ_{k}$ and $R_{I_{k}}|J_{k+1}=R_{I_{k-1}}^2|J_{k+1}$;
\item When $k\ge 2$ is even, $R_{I_{k}}|J_{k+1}=R_{I_{k-1}}|J_{k+1}$.
\end{enumerate}
Moreover, for each $\ell>1$, there is $a\in (0,1)$ such that the unimodal
\begin{equation}\label{eqn:goodmap}
f(x)=a(1-|2x-1|^\ell)
\end{equation}
has the above properties.
\end{prop}

Let $r_0=3$, $t_0=2$ and for each $k\ge 0$, define inductively
\begin{equation}\label{eqn:recr}
r_{k+1}=r_k+t_k, \text{ for all } k=0,1,\ldots,
\end{equation}
and
\begin{equation}\label{eqn:rect}
t_{k+1}=\left\{
\begin{array}{ll}
r_k & \text{ if } k \text{ is odd},
\\
r_{k+1} & \text{ if } k \text{ is even}.
\end{array}
\right.
\end{equation}
In fact, $r_k$ and $t_k$ are the return time of $c$ and $R_{I_k}(c)$ to $I_k$ for a map satisfying the properties (1)-(6).
\begin{proof} It suffices to construct a continuous unimodal map with the properties (1)-(6). In fact, each $\ell>1$, $x\mapsto a(1-|2x-1|^\ell)$, $0\le a\le 1$, is a full family and thus the existence of a map of the form (\ref{eqn:goodmap}) follows. See \cite[Section II.4]{MS}.

To construct such a continuous unimodal map, we argue as in \cite[Section 2.1]{NS}. We shall first construct inductively a sequence of unimodal maps $f_n:[0,1]\to [0,1]$, $n=0,1,\ldots$, with the same turning pint $c$, such that the intervals $I_0^{f_n}, I_1^{f_n},\ldots, I_{n+1}^{f_n}$, $J_1^{f_n}, J_2^{f_n}, \ldots, J_{n+1}^{f_n}$, $\widehat{J}_1^{f_n}, \widehat{J}_2^{f_n},\ldots, \widehat{J}_{n+1}^{f_n}$ are well defined and such that the properties (1)-(6) hold for $f=f_n$ and $k\le n$. Moreover, our construction satisfies
\begin{enumerate}
\item[(a)] $f_{n+1}=f_{n}$ for $x\in [0,1]\setminus I_{n+1}$ for all $n\ge 0$;
\item[(b)] $I_k^{f_n}=I_k^{f_{n-1}}$ and $J_k^{f_n}=J_k^{f_{n-1}}$ for all $n\ge 1$ and $k\le n$;
\item[(c)] $2|I_{n+2}^{f_{n+1}}|\le |I_{n+1}^{f_{n}}|$;
\item[(d)] $2|f_{n+1}(I_{n+2}^{f_{n+1}})|\le |f_n(I_{n+1}^{f_n})|$.
\end{enumerate}

For the starting step, we take $f_0$ to be an arbitrary unimodal map for which the critical point $c$ satisfies $f_0(c)>f_0^4(c)>c=f_0^5(c)> f_0^3(c)>f_0^2(c)$.  It is straightforward to check that for this map $f_0$, the return time of $c$ to $I_0^{f_0}=(\hat{q}^{f_0}, q^{f_0})$ is equal to $3$ and the return time of $f_0^3(c)$ to $I_0^{f_0}$ is $2$, so $I_1^{f_0}$ and $J_1^{f_0}$ are well-defined and disjoint. Moreover, the map $f_0^3$ is monotone increasing on the right component of $I_1^{f_0}\setminus\{c\}$ and thus $R_{I_0^{f_0}}^{f_0}(I_1^{f_0})\supset I_1^{f_0}\ni c$. The properties (4)-(6) are null in this case.

For the induction step, assuming that $f_n$ is defined such that the properties (1)-(6) holds for $f=f_n$ and $k\le n$, we shall modify the map $f_n$ on the interval $I_{n+1}^{f_n}$.
To be definite, let us assume that $n$ is even.
Note that $r_n$ is the return time of $c$ into $I_n^{f_n}$ and $t_n$ is the return time of $f_n^{r_n}(c)$ into $I_n^{f_n}$.
Let $U_{n+1}=(f_n^{r_n}|I_{n+1}^{f_n})^{-1}(J_{n+1}^{f_n})$ and let $V_{n+1}=(f_n^{r_n}|I_{n+1}^{f_n})^{-1}(\hJ_{n+1}^{f_n})$. Then $U_{n+1}\ni c$ is an interval, $V_{n+1}$ has two components and $U_{n+1}\cap V_{n+1}=\emptyset$. Moreover, each component $V_{n+1, i}$, $i=1,2$, of $V_{n+1}$ is mapped homeomorphically onto $I_n^{f_n}$ by $f_n^{r_n+t_n}=f_n^{r_{n+1}}$, so it contains an interval $W_{n+1, i}$ such that $f_n^{r_{n+1}}$ maps $W_{n+1,i}$ homeomorphically onto $I_{n+1}^{f_n}$.
Let $K_{n+1}$ be the component of $f_n^{-(r_{n+1}-1)}(I_n^{f_n})$ which contains $f_n(c)$. Then $f_n^{r_{n+1}-1}$ maps $K_{n+1}$ homeomorphically onto $I_{n}^{f_n}$ and $f_n(U_{n+1})\subset K_{n+1}$, $f_n(\partial U_{n+1})\subset \partial K_{n+1}$. Let $K_{n+1}'$ denote the subinterval of $K_{n+1}$ which is mapped onto $I_{n+1}^{f_n}$. Let $\varphi_n: K_{n+1}\to K_{n+1}$ be an orientation preserving homeomorphism such that $f_n^{r_{n+1}-1}(\varphi_n(f_n(c)))$ is contained in a component $W_{n+1,i_1}$ for some $i_1\in \{1,2\}$  and such that $f_n^{r_{n+1}-1}(\varphi_n(f_n(U_{n+1})))$ contains $W_{n+1, i_2}$ for $i_2\in \{1,2\}\setminus \{i_1\}$. Note that we can choose $\varphi_n$ such that $\varphi_n^{-1}(K_{n+1}')$ is deep inside $K_{n+1}$. Define $f_{n+1}=f_n$ outside $U_{n+1}$ and $f_{n+1}=\varphi_n\circ f_n$ on $U_{n+1}$. Then $f_{n+1}$ is the map which we looked for. Indeed, $J_{n+2}^{f_{n+1}}=W_{n+1, i_1}$ and $I_{n+2}^{f_{n+1}}=(f_n|U_{n+1})^{-1} (\varphi_n^{-1}(K_{n+1}'))$ can be made much smaller than $I_{n+1}^{f_n}$.

So these unimodal maps $f_n$ have been constructed. The properties (a) and (d) imply that $\{f_n\}$ is equi-continuous, while (a) and (c) implies that $f_n(x)$ is eventually constant for each $x\in [0,1]\setminus\{c\}$. Thus $f_n$ converges to a continuous unimodal map $f$ as $n\to\infty$.
By continuity, $f^{r_n}(c)=\lim_{m\to\infty} f_m^{r_n}(c)$ enters the closure of $I_{n}:=I_{n}^{f_{n-1}}$, so $c$ is a recurrent critical point, which implies then $f^i(c)$ is disjoint from the boundary of $I_n$ for any $i, n\ge 0$. It is then easily verified that the conditions (1)-(6) hold for $f$ by continuity.
\end{proof}

Now let us fix a unimodal map (\ref{eqn:goodmap}) with the properties (1)-(6).
We first show that
\begin{prop} \label{prop:ec}
The map $f:\omega(c)\to \omega(c)$ is topologically conjugate to an adding machine.
\end{prop}
\begin{proof}
Let $V_k=I_k\cup J_k\cup\hJ_k$ when $k$ is odd, and $V_k=I_k\cup J_k$ when $k$ is even. By induction, it is easy to see that for each $0\le l\le k$,
\begin{itemize}
\item $f^j(I_{k+1})\cap I_l\subset V_{l+1}$ holds for $0\le j< r_k$;
\item $f^j(J_{k+1})\cap I_l=f^j(\hJ_{k+1})\cap I_l\subset V_{l+1}$ holds for $0\le j< t_k$.
\end{itemize}
It follows that $$\omega(c)\cap I_k \subset \overline{V_{k+1}}.$$
Applying Lemma~\ref{lem:eqvcon} to $T_n=I_{2n}$, $T_n'=I_{2n+1}$, $Q_n=J_{2n+1}$ and $\hQ_n=\hJ_{2n+1}$, we conclude that $f:\omega(c)\to\omega(c)$  is topologically conjugate to an adding machine.
\end{proof}
To show that when $\ell>1$ is large enough, the map has a wild attractor, we shall apply ~\cite[Theorem 6.1]{B}.
 To this end, we shall recall a combinatorial language, called {\em kneading map}, which was used there.
Assume $f\in \mathcal{A}_*$.
The {\em closest precritical points} $z_k$ and {\em cutting times} $S_k$ are defined as follows: $S_0 := 1$,
$z_0 := f^{-1}(c) \cap (0, c)$. Inductively,
$$S_{k+1} := \min\{n > S_k:  f^{-n}(c)\cap (z_k, c)\not=\emptyset\},$$
and
$$z_{k+1} := f^{-S_{k+1}}(c)\cap (z_k, c).$$
For each $k=1,2,\ldots$, $f^{S_{k+1}}$ is monotone on the interval $[z_k, c]$. So $f^{S_{k+1}-S_k}$ maps $[c, c_{S_k}]$ monotonically onto an interval containing $c$. Thus $S_{k+1}-S_k$ is also a cutting time. This implies that there is an integer $Q(k+1)$ such that
$$S_{k+1}=S_k+S_{Q(k+1)}.$$
Define also $Q(0)=0$. The function $k\mapsto Q(k)$, $k=0,1,\ldots$ is called the {\em kneading map} of $f$.

\begin{lem} The cutting times of $f$ not smaller than $r_1$ are the following:
$$r_1,  r_1+r_0, r_2, r_3, r_2+r_3, r_4, r_5, \ldots.$$
\end{lem}
\begin{proof} For each $k=1,2,\ldots$, $f^{r_k}$ maps each component of $I_{k+1}\setminus \{c\}$ monotonically onto an interval containing $c$, thus $r_k$ is a cutting time.
If $k\ge 2$ is even, then $f^{t_k}=f^{r_{k-1}}$ is monotone on $[c_{r_k}, c]\subset I_k$, hence there is no cutting time between $r_k$ and $r_k+t_k=r_{k+1}$.
Now assume that $k$ is odd. We need to show that $r_k+r_{k-1}$ is the only cutting time between $r_k$ and $r_{k+1}$.
To this end, let $s$ be the cutting time next to $r_k$.  So $s$ is the smallest integer with $s>r_k$ and such that $f^{s-r_k}$ maps $[c_{r_k},c]$ monotonically onto an interval containing $c$ in the interior. Since $c_{r_k}\in J_{k+1}$, and $f^{r_{k-1}}$ is monotone on $[c_{r_k},c]$, so $s\ge r_k+r_{k-1}$. Note that there is a component $K_{k+1}$ of $f^{-r_{k-1}}(I_k)$ between $I_{k+1}$ and $J_{k+1}$, and $f^{r_{k-1}}$ maps $K_{k+1}$ monotonoically onto $I_k$. Thus $s=r_k+r_{k-1}$. Finally since $c_{r_k+r_{k-1}}\in I_k$, $f^{t_{k-1}}$ is monotone on $[c_{r_k+r_{k-1}},c]$, there is no cutting time between $r_k+r_{k-1}$ and $r_k+r_{k-1}+t_{k-1}=r_{k+1}$.
\end{proof}

\begin{lem}\label{lem:bruincond}
There exist positive integers $k_1$ and $N$ such that when $k\ge k_1$ we have
$$Q(k+1)\ge Q^2(k)+1,$$
and $$k-Q(k)\le N.$$
\end{lem}
\begin{proof} The previous lemma implies that there exists $k_0\ge 1$ such that for each $m=0,1,\ldots$,
$$S_{k_0+3m+1}=r_{2m+1}, S_{k_0+3m+2}=r_{2m+1}+r_{2m}, S_{k_0+3m+3}=r_{2m+2}.$$
Using (\ref{eqn:recr}) and (\ref{eqn:rect}), we easily find
\begin{align*}
S_{k_0+3m+2}-S_{k_0+3m+1}& =r_{2m}=S_{k_0+3m}\\
S_{k_0+3m+3}-S_{k_0+3m+2}& = t_{2m+1}-r_{2m}= r_{2m-1}=S_{k_0+3m-2},\\
S_{k_0+3m+4}-S_{k_0+3m+3}& = r_{2m+3}-r_{2m+2}= r_{2m+1}= S_{k_0+3m+1},
\end{align*}
which implies that when $k$ is large enough,
$$Q(k)=\left\{
\begin{array}{ll}
k-5 & \mbox{ if } k-k_0\equiv 0 \mod 3£»\\
k-3 & \mbox{ if } k-k_0\equiv 1 \mod 3£»\\
k-2 &\mbox{ if } k-k_0\equiv 2\mod 3£¬
\end{array}
\right.
$$
and hence
$$Q^2(k)=\left\{
\begin{array}{ll}
k-8 & \mbox{ if } k-k_0\equiv 0\mod 3£»\\
k-6 &\mbox{ if } k-k_0\equiv 1\mod 3£»\\
k-7 &\mbox{ if } k-k_0\equiv 2\mod 3.
\end{array}
\right.
$$
Thus the lemma holds with $N=5$ and $k_1$ large enough.
\end{proof}
\begin{cor} \label{cor:wilda}
When $\ell>1$ is sufficiently large, $f$ has a wild attractor.
\end{cor}
\begin{proof} This follows directly from Lemma~\ref{lem:bruincond} by \cite[Theorem 6.1]{B}.
\end{proof}

\begin{proof}[Proof of Theorem~\ref{thm:ec}]
Let $f$ be a unimodal map of the form (\ref{eqn:goodmap}) which satisfies the properties (1)-(6) as in Proposition~\ref{prop:map}. By (\ref{prop:ec}), $f:\omega(c)\to\omega(c)$ is equicontinuous. By Corollary~\ref{cor:wilda}, $\omega(c)$ is a wild attractor provided that $\ell$ is sufficiently large.
\end{proof}

\end{document}